\documentclass{amsart}
\usepackage{geometry}                
\usepackage{graphicx}
\usepackage{amssymb}
\usepackage{epstopdf,amsmath}

\newtheorem{prop}{Proposition}[section]
\newtheorem{lemma}[prop]{Lemma}
\newtheorem{sublemma}[prop]{Sublemma}
\newtheorem{cor}[prop]{Corollary}
\newtheorem{theorem}[prop]{Theorem}
\newtheorem{conjecture}[prop]{Conjecture}
\newtheorem{defn}[prop]{Definition}

\DeclareMathOperator{\Avg}{Avg}
\DeclareMathOperator{\Deg}{Deg}
\DeclareMathOperator{\Dim}{Dim}
\DeclareMathOperator{\Sign}{Sign}
\DeclareMathOperator{\Err}{Err}
\DeclareMathOperator{\Dist}{Dist}
\DeclareMathOperator{\Vol}{Vol}
\DeclareMathOperator{\HausDist}{Haus Dist}

\numberwithin{equation}{section}

\newcommand{\dt}{\delta}
\newcommand{\dm}{\delta_m}
\newcommand{\dl}{\delta_l}

\newcommand{\CC}{\mathbb{C}}
\newcommand{\eps}{\epsilon}

\newcommand{\ZZ}{\mathbb{Z}}
\newcommand{\RR}{\mathbb{R}}
\newcommand{\EP}{E}
\newcommand{\PsiP}{\Psi_{par}}
\newcommand{\EG}{T}

\newcommand{\PsiG}{\Psi}
\DeclareMathOperator{\Angle}{Angle}
\DeclareMathOperator{\Fcn}{Fcn}
\DeclareMathOperator{\BL}{BL}
\newcommand{\BLp}{\BL^p}
\newcommand{\BLpkA}{\BL^p_{k,A}}
\newcommand{\BLpkAt}{\BL^p_{k,A/2}}

\newcommand{\ftv}{f_{\theta,v}}
\newcommand{\gtv}{g_{\theta,v}}
\newcommand{\tgtv}{\tilde g_{\ttv}}
\newcommand{\Ttv}{T_{\theta,v}}
\newcommand{\tTtv}{T_{\tilde \theta, \tilde v}}
\newcommand{\tv}{{\theta,v}}
\newcommand{\ttv}{{\tilde \theta,\tilde v}}
\newcommand{\tx}{\tilde x}
\newcommand{\tf}{\tilde f}
\newcommand{\tg}{\tilde g}
\newcommand{\tih}{\tilde h}

\newcommand{\tftv}{\tilde f_{\ttv}}
\newcommand{\tith}{\tilde \theta}
\newcommand{\tiv}{\tilde v}

\newcommand{\pkm}{\bar p(k,m)}
\newcommand{\pkn}{\bar p(k,n)}
\newcommand{\TTT}{\mathbb{T}}

\DeclareMathOperator{\RD}{RapDec}

\newcommand{\tNZR}{\tilde N_{R^{1/2 + \dm}} (Z) \cap B_R}
\newcommand{\NZR}{N_{R^{1/2 + \dm}} (Z) \cap B_R}

\newcommand{\NY}{N_{R^{1/2 + \dm}}(Y)}

\title{Restriction estimates using polynomial partitioning II}

\author{Larry Guth}

\begin{document}

\begin{abstract} We improve the estimates in the restriction problem in dimension $n \ge 4$.  To do so, we establish a weak version of a $k$-linear restriction estimate for any $k$.  The exponents in this weak $k$-linear estimate are sharp for all $k$ and $n$.
\end{abstract}

\maketitle


\section{Introduction}

This paper gives improved restriction estimates for the paraboloid in high dimensions.   Recall that the extension operator for the paraboloid can be written in the form

\begin{equation} \label{defEP} \EP f(x) := \int_{B^{n-1}} e^{ i ( x_1 \omega_1 + ... +  x_{n-1} \omega_{n-1} + x_n | \omega |^2) } f(\omega) d \omega, \end{equation}

\noindent where $B^{n-1}$ denotes the unit ball in $\RR^{n-1}$ and $x \in \RR^n$.  Stein \cite{Ste} conjectured that the extension operator should obey the inequality 

\begin{equation}  \label{restLp} \| \EP f \|_{L^p(\RR^n)} \lesssim \| f \|_{L^p(B^{n-1})} \end{equation}

\noindent for all $p > 2 \cdot \frac{n}{n-1}$.  We prove new partial results towards this conjecture in dimension $n \ge 4$.  

\begin{theorem} \label{L^pEP} For $n \ge 2$, the operator $\EP$ obeys the estimate (\ref{restLp}) if 

\begin{equation} \label{prangeodd} p >  2 \cdot \frac{3n + 1}{3n - 3} \textrm{ for } n \textrm{ odd }.
\end{equation}

\begin{equation} \label{prangeeven}  p > 2 \cdot \frac{3n + 2}{3n - 2} \textrm{ for } n \textrm{ even }.
\end{equation}

\end{theorem}

The best previous estimates for the problem were proven by Tao \cite{T2} for $n=4$ and by Bourgain and the author \cite{BG} for $n \ge 5$.  
 For $n = 4$, the conjecture is that (\ref{restLp}) holds for $p > 2 \frac{2}{3}$.  Theorem \ref{L^pEP} gives the range $p > 2.8$, and the best previous estimate was $p > 3$.  Asymptotically for large $n$, the conjecture is that  (\ref{restLp}) holds for $p$ bigger than the lower bound $ \frac{2n}{n-1} = 2 + 2 n^{-1} + O(n^{-2})$.  The lower bound for $p$ in Theorem \ref{L^pEP} is  $2 + \frac{8}{3} n^{-1} + O(n^{-2})$, and the lower bound in the best previous estimate was $2 + 3 n^{-1} +  O(n^{-2})$.  

The new ingredient of our argument has to do with algebraic structure.  Roughly speaking, the argument shows that if $\| \EP f \|_{L^p}$ is large, then the region where $| \EP f|$ is large must be organized into thin neighborhoods of low degree algebraic varieties.  Exploiting this structure leads to improved bounds on $\| \EP f \|_{L^p}$.    We find this algebraic structure using the tool of polynomial partitioning, which was introduced by Katz and the author in \cite{GK}.

Polynomial partitioning was first applied to the restriction problem in \cite{Gu4}, which gave the best current restriction estimate in dimension 3.  In this paper we combine that approach with ideas from \cite{BG}.  Besides making incremental progress on the restriction conjecture, the methods in this paper are related to sharp results for some other problems in the field, which we describe in the next two subsections.

\subsection{Related work} 

In this subsection, we describe two papers which build on this one, and adapt the methods to other problems.

In \cite{GHI}, Hickman, Iliopoulou and the author generalize Theorem \ref{L^pEP} to the setting of Hormander-type operators with positive-definite phase.  For this more general class of operators, the estimates are sharp up to the endpoint.  Hormander-type operators with positive-definite phase can be thought of as small perturbations of the extension operator $E$.  To formulate this precisely, we first write $\EP$ in a slightly different form.  We define the phase function 

\begin{equation} \label{defPsiP} \PsiP(y, \omega) := y_1 \omega_1  + ... + y_{n-1} \omega_{n-1} + y_n |\omega|^2. \end{equation}

\noindent We restrict $x$ to a ball of radius $R$, $B^n_R$.  For 
$x \in B^n_R$, we can write $\EP f(x)$ in the form

$$ \EP f(x) = \int_{B^{n-1}} e^{i R \PsiP(\frac{x}{R}, \omega)} f(\omega) d \omega. $$

\noindent We think of $\PsiP$ as a function from $B^n \times B^{n-1}$ to $\RR$.  We now consider other phase functions $\PsiG(y, \omega)$, which are small $C^\infty$ perturbations of $\PsiP(y, \omega)$ on $B^n \times B^{n-1}$.  For each such phase function $\PsiG$, and each scale $R$, we define an operator

\begin{equation} \label{defEG} \EG f(x) = \int_{B^{n-1}} e^{i R \PsiG(\frac{x}{R}, \omega)} f(\omega) d \omega. \end{equation}

\noindent Hormander introduced this type of operator in \cite{H}.  As we said above, operators of the form (\ref{defEG}) can be thought of as small perturbations of the extension operator $\EP$.  Hormander raised the question whether all such operators obey the $L^p$ bounds conjectured to hold for $\EP$, and he proved that this is the case when $n=2$.  But it turns out to be false for all $n \ge 3$.  A counterexample was found by Wisewell \cite{Wi} (cf. also \cite{BG}).  These counterexamples build on a well-known counterexample of Bourgain from \cite{B1} for a related but slightly different problem.  These counterexamples are surprising because they show that a $C^\infty$ small perturbation of the phase function can cause a major change in the behavior of the operator.  In this context, it is reasonable to ask about the best $L^p$ estimates that hold for all operators of the form (\ref{defEG}) -- the best estimates that are robust to such small perturbations.  Hormander \cite{H} answered this question in dimension $n=2$, and Lee \cite{L} did so in dimension $n=3$.  The paper \cite{GHI} does so for all $n$.  It shows that

$$ \| \EG f \|_{L^p(B^n_R) } \lesssim \| f \|_{L^p(B^{n-1})} \textrm{ for the range of $p$ in Theorem \ref{L^pEP}}. $$

\noindent The counterexamples from \cite{Wi} and \cite{BG} show that, up to the endpoint, this is the sharp range of $p$ in every dimension.  

In another direction, in \cite{OW}, Ou and Wang adapt the methods here to the case of the cone.  They prove the sharp range of restriction estimates for the cone in dimension $n \le 5$.  Previously, Wolff \cite{W1} proved the sharp range of restriction estimates for the cone in dimension $n \le 4$.

\subsection{$k$-linear estimates and $k$-broad estimates}

Multilinear estimates have played a key role in the recent developments in restriction theory.  Our main new result, which leads to Theorem \ref{L^pEP}, is a weaker version of a $k$-linear restriction estimate, which we call a $k$-broad estimate.  The exponents in our $k$-broad estimate are sharp for all $k$.  We recall some background on multilinear estimates and then formulate this new result.

We begin by recalling the wave packet decomposition.  Suppose we want to study $\EP f$ on a large ball $B_R \subset \RR^n$.  We decompose the domain $B^{n-1}$ into balls $\theta$ of radius $R^{-1/2}$.  Then we decompose $f$ in the form 

$$f = \sum_{\theta, v} f_{\theta, v}$$

\noindent where $f_{\theta, v}$ is supported in $\theta$ and has Fourier transform essentially supported in a ball around $v$ of radius $R^{1/2}$.  In the sum, $\theta$ ranges over our set of finitely overlapping balls covering $B^{n-1}$ and $v$ ranges over $R^{1/2} \ZZ^{n-1}$.  For each pair $(\theta, v)$, the restriction of $\EP f_{\theta,v}$ to $B_R$ is essentially supported on a tube $T_{\theta, v}$ with radius $R^{1/2}$ and length $R$.  The direction of this tube depends only on $\theta$, and we denote it by $G(\theta) \in S^{n-1}$.  We call $\EP f_{\theta, v}$ a wave packet.

We can now describe multilinear restriction estimates.  Given subsets $U_1, ...., U_k \subset B^{n-1}$, we say that they are transverse if, for any choice of $\theta_j \subset U_j$, the directions $G(\theta_1), ..., G(\theta_k)$ are quantitatively transverse in the sense that

\begin{equation} \label{defktrans} | G(\theta_1) \wedge ... \wedge G(\theta_k) | \gtrsim 1. \end{equation}

Building on important work of Wolff \cite{W1}, Tao \cite{T2} proved a sharp bilinear estimate for the extension operator $\EP$.

\begin{theorem} \label{2linear} (2-linear restriction, \cite{T2}) If $U_1, U_2 \subset B^{n-1}$ are transverse, and $f_j$ is supported in $U_j$, then

\begin{equation}  \label{2linest} \left\| \prod_{j=1}^2 | \EP f_j |^{1/2} \right\|_{L^p(B_R)} \lesssim R^\eps \prod_{j=1}^2 \| f_j \|_{L^2(B^{n-1})}^{1/2} \end{equation}

for $ p \ge 2 \cdot \frac{n + 2}{n}$.  

\end{theorem}

By an argument of Tao, Vargas, and Vega, \cite{TVV}, this bilinear estimate implies that $\| \EP f \|_{L^p(B_R)} \lesssim R^\eps \| f \|_{L^p(B^{n-1})}$ in the same range $p \ge 2 \cdot \frac{n + 2}{n}$.  The $\eps$-removal theorem (\cite{T1}), then implies $\| \EP f \|_{L^p(B_R)} \lesssim \| f \|_{L^p(B^{n-1})}$ for all $p > 2 \cdot \frac{n+2}{n}$. 

A few years after the bilinear results, Bennett, Carbery, and Tao \cite{BCT} proved a sharp $n$-linear estimate for $\EP$.  

\begin{theorem} \label{nlinear} (n-linear restriction, \cite{BCT}) If $U_1, ..., U_n \subset B^{n-1}$ are transverse, and $f_j$ is supported in $U_j$, then

\begin{equation}  \label{nlinest} \left\| \prod_{j=1}^n | \EP f_j |^{1/n} \right\|_{L^p(B_R)} \lesssim R^\eps \prod_{j=1}^n \| f_j \|_{L^2(B^{n-1})}^{1/n} \end{equation}

for $ p \ge 2 \cdot \frac{n}{n - 1}$.  

\end{theorem}

This theorem is important and remarkable in part because it involves the sharp exponent for the restriction problem: $p > 2 \cdot \frac{n}{n-1}$.  The paper \cite{BG} gives a technique to exploit multilinear restriction estimates in order to get improved estimates on the original restriction problem.  Since then, multilinear restriction has had many applications, including the striking recent work of Bourgain and Demeter on decoupling (see \cite{BD} and many followup papers).

Given this $2$-linear estimate and this $n$-linear estimate, it is natural to try to prove a $k$-linear estimate for all $2 \le k \le n$ which would include these two estimates as special cases.  Here is what looks to me like the natural conjecture, which I first learned from Jonathan Bennett.

\begin{conjecture}  \label{klinear} (k-linear restriction) If $U_1, ..., U_k \subset B^{n-1}$ are transverse, and $f_j$ is supported in $U_j$, then

\begin{equation}  \label{klinest} \left\| \prod_{j=1}^k | \EP f_j |^{1/k} \right\|_{L^p(B_R)} \lesssim R^\eps \prod_{j=1}^k \| f_j \|_{L^2(B^{n-1})}^{1/k} \end{equation}

for $ p \ge \pkn := 2 \cdot \frac{n +  k}{n + k - 2}$.  

\end{conjecture}  

Having the full range of $k$-linear estimates available would improve the results from \cite{BG}.  Combining Conjecture \ref{klinear} with the method from \cite{BG} would give $\| \EP f \|_{L^p} \lesssim \| f \|_{L^p}$ for exactly the range of $p$ in Theorem \ref{L^pEP}.  

For $3 \le k \le n-1$, Conjecture \ref{klinear} is open.  In \cite{Be} and \cite{Be2}, Bejenaru proves multilinear estimates for certain curved hypersurfaces, but not including the paraboloid.  The surfaces he considers are foliated by $(k-1)$-planes and they are curved in the transverse directions in an appropriate sense.  The main new result of this paper is a weak version of Conjecture \ref{klinear}, which we call a $k$-broad restriction inequality.  To motivate this inequality, let us recall the approach from \cite{BG} for deducing linear estimates from multilinear ones.  

We decompose $B^{n-1}$ into balls $\tau$ of radius $K^{-1}$, where $K$ is a large constant.  This decomposition is much coarser than the decomposition into balls $\theta$ of radius $R^{-1/2}$.  We write $f = \sum_\tau f_\tau$ where $f_\tau$ is supported in $\tau$.  Next we subdivide $B^n_R$ into much smaller balls.  In \cite{BG}, we used balls of radius $K$, but it will be slightly more convenient here to use balls of radius $K^2$.  For each $B_{K^2} \subset B_R$, we consider $ \int_{B_{K^2}} | \EP f_\tau|^p$ for each $\tau$.  We say that $\tau$ contributes significantly to $B_{K^2}$ if

$$ \int_{B_{K^2}} | \EP f_\tau |^p \gtrsim K^{-10n} \int_{B_{K^2}} | \EP f|^p. $$

\noindent We let $S(B_{K^2})$ denote the set of $\tau$ which contribute significantly to $B_{K^2}$.   Now we break the balls $B_{K^2}$ into two classes.  We label a ball $B_{K^2}$ as $k$-transverse if there are $k$ significant $\tau$'s which are $k$-transverse in the sense above.  We label a ball $B_{K^2}$ as $k$-non-transverse otherwise.  A $k$-linear restriction estimate gives a good bound for the integral of $| \EP f |^p$ over the union of all of the $k$-transverse balls.  The paper \cite{BG} then gives an inductive argument to control the contribution from the $k$-non-transverse balls.

This inductive argument can be described most cleanly using the language of decoupling.  Building on \cite{BG}, Bourgain proved a decoupling theorem in \cite{B4} which implies that for each $k$-non-transverse ball $B_{K^2}$, 

$$ \| \EP f \|_{L^p(B_{K^2})}^2 \lesssim \sum_{\tau \in S(B_{K^2})} \| \EP f_\tau \|_{L^p(B_{K^2})}^2$$

\noindent for a certain range of $p$ which covers the exponents we study.  For a $k$-non-transverse ball $B_{K^2}$, all the significant $\tau$ have direction $G(\tau)$ within $K^{-1}$ of some $(k-1)$-plane.  In particular $|S(B_{K^2})| \lesssim K^{k-2}$.  Using this bound and Holder's inequality we get

\begin{equation} \label{narrowballs} \int_{B_{K^2}} |\EP f|^p \lesssim K^\alpha \sum_{\tau \in S(B_{K^2})} \int_{B_{K^2}} | \EP f_\tau|^p. \end{equation}

\noindent with $\alpha = (K^{k-2})^{\frac{p-2}{2}}$, which is optimal.  Summing this inequality over all the $k$-non-transverse balls gives

$$ \int_{\bigcup \textrm{$k$-non-transverse balls} } | \EP f|^p \lesssim K^\alpha \sum_\tau \int_{B_R} | \EP f_\tau|^p. $$

\noindent The right-hand side may then be controlled by induction on scales.

To make this strategy work, we do not need a full $k$-linear bound.  The decoupling estimate that we used to control the $k$-non-transverse balls applies whenever the significant $\tau$ for a ball $B_{K^2}$ all have directions $G(\tau)$ lying within the $O(K^{-1})$-neighborhood of $O(1)$ $(k-1)$-planes.  We call such a ball $k$-narrow.  To get the argument to work, we only need to bound
$\int |\EP f|^p$ over the remaining balls -- the $k$-broad balls.  

Here is a little notation so that we can state our $k$-broad bound precisely.  We let $G(\tau) = \cup_{\theta \subset \tau} G(\theta)$.  The set $G(\tau) \subset S^{n-1}$ is a spherical cap with radius $\sim K^{-1}$, representing the possible directions of wave packets in $\EP f_\tau$.  If $V \subset \RR^n$ is a subspace, then we write $\Angle( G(\tau), V)$ for the smallest angle between any non-zero vectors $v \in V$ and $v' \in G(\tau)$.  For each ball $B_{K^2} \subset B_R$, we consider $\int_{B_{K^2}} | \EP f_\tau |^p$ for every $\tau$.    To define the $k$-broad norm, we discount the contributions of $f_\tau$ with $G(\tau)$ lying near to a few $(k-1)$-planes, and we record the largest remaining contribution.  More formally, for a parameter $A$, we define

\begin{equation} \label{defmuEf} \mu_{\EP f} (B_{K^2}) := \min_{V_1, ..., V_A \textrm{ $(k-1)$-subspaces of } \RR^n} \left( \max_{\tau: \Angle(G(\tau), V_a) > K^{-1} \textrm{ for all } a } \int_{B_{K^2}} | \EP f_\tau|^p \right). \end{equation}

We can now define the $k$-broad part of $\| \EP f \|_{L^p(B_R)}$ by:

\begin{equation} \label{defBroadL^p}
\| \EP f \|_{\BLpkA(B_R)}^p := \sum_{B_{K^2} \subset B_R} \mu_{\EP f} (B_{K^2}) . 
\end{equation}

Our main new result is an estimate for this $k$-broad norm.

\begin{theorem} \label{broadEP} For any $ 2 \le k \le n$, and any $\eps > 0$, there is a large constant $A$ so that the following holds (for any value of $K$):

\begin{equation} \label{broadest} \| \EP f \|_{\BLpkA(B_R)} \lesssim_{K, \eps} R^\eps \| f \|_{L^2(B^{n-1})}, \end{equation}

for $ p \ge \pkn = 2 \cdot \frac{n +  k}{n + k - 2}$.  

\end{theorem}

The range of $p$ in Theorem \ref{broadEP} is sharp for all $k$ and $n$.  Using the method from \cite{BG} outlined above, Theorem \ref{broadEP} implies the restriction estimate Theorem \ref{L^pEP}.  

We should mention that $\BLpkA$ is not literally a norm, but it has some similar properties, which is why we use the norm notation.  In particular, $\BLpkA$ obeys the following weak version of the triangle inequality:  if $f = g+h$, then

\begin{equation} \label{introquasitriangle} \| E f \|_{\BLpkA (B_R)} \lesssim \| Eg \|_{\BL^p_{k, A/2}(B_R)} + \| E h \|_{\BL^p_{k, A/2}(B_R)}. \end{equation}

\noindent The reason for introducing the parameter $A$ is to use this version of the triangle inequality.  If we choose $A(\eps)$ very large, then we can effectively use the triangle inequality $O_\eps(1)$ times during our proof, and so $\BLpkA$ behaves almost like a norm.   We were not able to prove Conjecture \ref{klinear}, and the main issue is that in the true $k$-linear setting, we do not have a substitute for this triangle inequality.

\subsection{Examples}

To help digest Theorem \ref{broadEP}, we describe a couple examples.  These examples show that the range of exponents $p$ in Theorem \ref{broadEP} is sharp.

In one example, the wave packets $\EP \ftv$ concentrate in the $R^{1/2}$-neighborhood of a $k$-plane.  We denote this neighborhood by $W$.  Each wave packet $\EP \ftv$ has $| \EP \ftv(x) | \sim 1$ on the tube $\Ttv$, and rapidly decaying outside $\Ttv$.  
It is not hard to arrange that each point in the slab $W$ lies in many wave packets $\Ttv$, pointing in many directions within the $k$-plane.  For each ball $B_{K^2}$ in the slab, only a tiny fraction of the wave packets through this ball lie near to any $(k-1)$-plane.  In this scenario, $\| \EP f \|_{\BLpkA(B_R)} \sim \| \EP f \|_{L^p(B_R)}$.  We can also arrange that $| \EP f(x) |^2 \sim \sum_\tv | \EP \ftv(x) |^2$ at most points $x$ by replacing $\ftv$ by $\pm \ftv$ with independent random sings.  We can distribute the wave packets evenly so that $| \EP f(x)|$ is roughly constant on the slab.  Moreover, by a standard orthogonality argument $\| \EP f \|_{L^2(B_R)} \sim R^{1/2} \| f \|_{L^2}$.  Therefore, we get

$$ \frac{  \| \EP f \|_{\BLpkA(B_R)}}{ \| f \|_{L^2}} \sim  R^{1/2} \frac{ \| \EP f \|_{L^p(B_R)}}{ \| \EP f \|_{L^2(B_R)}} \sim  R^{1/2} | W |^{\frac{1}{p} - \frac{1}{2}}. $$

\noindent Since $|W| \sim R^k R^{(1/2)(n-k)}$, a short calculation shows that the ratio $ \| \EP f \|_{L^p(B_R)} / \| f \|_{L^2}$ is bounded for $p \ge \pkn$ and blows up for $p < \pkn$.  

But there are also more complicated sharp examples, coming from low degree algebraic varieties.  This type of example was first pointed out to me by Josh Zahl.  For instance, consider the quadric hypersurface $Z \subset \RR^4$ defined by

$$ (x_1/ R)^2 + (x_2/R)^2 - (x_3/R)^2 - (x_4/R)^2 = 1. $$

\noindent Each point of $Z$ lies in a 1-parameter family of lines in $Z$.  The union of the lines through a given point form a 2-dimensional cone.   For example, the point $(R, 0, 0, 0)$ lies in the cone defined by

$$x_1 = R; (x_2/R)^2 - (x_3/R)^2 - (x_4/R)^2 = 0. $$

\noindent If we take the $R^{1/2}$-neighborhoods of lines in $Z$, we can find many tubes in the $R^{1/2}$-neighborhood of $Z$.  We can now build an example like the one above using wave packets concentrated in the $R^{1/2}$-neighborhood of $Z$.  For each ball $B_{K^2}$ in this neighborhood, the wave packets through $B_{K^2}$ fill out a 2-dimensional cone, and very few of them lie near any 2-dimensional plane.  Therefore, $\| \EP f \|_{\BL^p_{3, A}(B_R)} \sim \| \EP f \|_{L^p(B_R)}$.  The rest of the discussion in the planar slab example applies here also, and so we see that this example is sharp for Theorem \ref{broadEP} in dimension 4 with $k=3$.  

For larger $n$, there are more variations on this example.  The dimension of the variety $Z$ in these examples is $k$.  
The degree of $Z$ may be larger than 2, although in the known examples it is always bounded by $C(n)$.  
Similar examples apply to Conjecture \ref{klinear}.  

These examples help to suggest that algebraic varieties could be relevant to Theorem \ref{broadEP}.  Polynomial partitioning is a tool that helps us to find and exploit the type of algebraic structure in these examples.  If we run through the proof of Theorem \ref{broadEP} on this type of example, the argument will find the variety $Z$.  

The new difficulty in this paper, compared with \cite{Gu4}, is that it is harder to find a $k$-dimensional variety for small $k$ then it is to find a hypersurface.  In the next section, we will describe the polynomial partitioning process and give a sense of the issues involved.

\subsection{A direction for further improvement}

The paper \cite{Gu4} applies polynomial partitioning to the restriction problem in 3 dimensions.  It proves an estimate which is stronger than Theorem \ref{L^pEP}, namely $\| \EP f \|_{L^p(\RR^3)} \lesssim \| f \|_{L^\infty}$ for $p > 3.25$.  
This estimate relies on one additional ingredient: an estimate for how many different $\theta$ can be represented by wave packets $T_{\theta,v}$ in the $R^{1/2}$-neighborhood of a low degree variety $Z$.  In the 3-dimensional case, recall that there are $\sim R$ balls $\theta \subset B^2$, each with radius $R^{-1/2}$.  Let $\Theta(Z)$ denote the set of $\theta$ so that at least one wave packet $T_{\theta, v}$ is contained in the $R^{1/2}$-neighborhood of $Z$.  Lemma 3.6 of \cite{Gu4} proves that if $Z$ is a 2-dimensional variety in $\RR^3$ of degree $\lesssim 1$, then $|\Theta(Z)| \lesssim_\eps R^{1/2 + \eps}$.  (If $Z$ is a 2-plane, then $|\Theta(Z)| \sim R^{1/2}$, and so the result says that the 
example of a plane is nearly the worst possible.)  

When $\| \EP f \|_{L^p}$ is large, the polynomial partitioning method locates algebraic pieces that contribute most of $\| \EP f \|_{L^p}$.  The bound for $|\Theta(Z)|$ gives a stronger estimate for the contribution of each such piece in terms of $\| f \|_{L^\infty}$ or $\| f \|_{L^p}$.  Note that for Hormander-type operators of positive definite phase in 3 dimensions, the estimate $\| \EG f \|_{L^p(B^3_R)} \lesssim R^\eps \| f \|_{L^\infty(B^2)}$ is false for all $p < 10/3$.  To prove the bound $\| \EP f \|_{L^p} \lesssim \| f \|_{L^\infty}$ for all $p > 3.25$, the argument from \cite{Gu4} has to distinguish $\EP$ from more general Hormander-type operators of positive-definite phase.
The bound on $|\Theta(Z)|$ is the step that does this.  In the Hormander case, $ \EG f_{\theta,v}$ is concentrated on a curved tube.  And in the counterexample from \cite{Wi} or \cite{BG}, there is a low-degree variety $Z$ whose $R^{1/2}$-neighborhood contains one such curved tube for every $\theta$.  

I have not been able to prove a good bound for $|\Theta(Z)|$ in higher dimensions.  Such a bound would lead to further improvements in the restriction exponents in high dimensions.  We will discuss this issue more in the final section of the paper.


\vskip5pt

{\bf Acknowledgements.}  I was supported by a Simons Investigator Award during this work.  I would also like to thank Marina Iliopoulou and Jongchon Kim for helpful comments on a draft of the paper.

\section{Sketch of the proof}

In this section, we sketch the proof of the $k$-broad estimate, Theorem \ref{broadEP}.  
We actually give two sketches.  The first sketch aims to show the main ideas of the argument.  The second sketch brings into play more of the technical issues, and it provides a detailed outline of the argument in the paper.

The proof begins with a wave packet decomposition.  We decompose the domain $B^{n-1}$ into balls $\theta$ of radius $R^{1/2}$.  We then decompose the function $f: B^{n-1} \rightarrow \CC$ as

$$ f = \sum_{\theta, v} f_{\theta, v}, $$

\noindent where $\ftv$ is supported on $\theta$ and the Fourier transform of $\ftv$ is essentially supported on a ball of radius $R^{1/2}$ around $v$.  In the sum, $v$ ranges over $R^{1/2} \ZZ^{n-1}$.  
On $B_R$, $\EP \ftv$ is essentially supported on a tube $T_{\theta, v}$ of radius $R^{1/2}$ and length $R$.  In addition, the functions $\ftv$ are essentially orthogonal.  In particular, we have

\begin{equation} \label{outorthog} \| f \|_{L^2}^2 \sim \sum_{\theta,v} \| \ftv \|_{L^2}^2. \end{equation}

Our goal is to prove that 

\begin{equation} \label{outgoal} \| \EP f \|_{\BLpkA (B_R)} \le C(\eps) R^\eps \| f \|_{L^2} \textrm{ for } p = \pkn := \frac{2(n+k)}{n + k -2}. \end{equation}

The proof will be by induction.  So we assume that (\ref{outgoal}) holds for balls of smaller radii, and in lower dimension.

Recall that

$$\| \EP f \|_{\BLpkA(B_R)}^p = \sum_{B_{K^2} \subset B_R} \mu_{\EP f} (B_{K^2}), $$

\noindent where $\mu_{\EP f}(B_{K^2})$ was defined in (\ref{defmuEf}).  We can extend $\mu_{\EP f}$ to be a measure on $B_R$, making it a constant multiple of the Lebesgue measure on each $B_{K^2}$.  In particular $ \mu_{\EP f}(B_R) = \| \EP f \|_{\BLpkA (B_R)}^p$. 

We now introduce polynomial partitioning.  We let $D$ be a large constant that we can choose later.  For a polynomial $P$ on $\RR^n$, we write $Z(P)$ for the zero set of $P$.  By Theorem 1.4 in \cite{Gu4}, there is a (non-zero) polynomial $P$ of degree at most $D$ on $\RR^n$, so that $\RR^n \setminus Z(P)$ is a disjoint union of $\sim D^n$ open cells $O_i$, and the measures $\mu_{\EP f}(O_i)$ are all equal. 

Next we consider how the wave packets $\EP \ftv$ interact with this partition.  We note that a line can cross $Z(P)$ at most $D$ times, and so a line can enter at most $D+1$ of the $\sim D^n$ cells $O_i$.  The tube $\Ttv$ can still enter many or all cells $O_i$, but it can only penetrate deeply into $D+1$ cells.  To make this precise, we define $W$ to be the $R^{1/2}$-neighborhood of $Z(P)$, and we define $O_i'$ to be $O_i \setminus W$.  If a tube $\Ttv$ enters $O_i'$, then the axis of $\Ttv$ must enter $O_i$, and so we get

\begin{equation} \label{outtubeDcells} \textrm{Each tube $\Ttv$ enters at most $D+1$ cells $O_i'$.}
\end{equation}

We now have 

$$ \mu_{\EP f}(B_R) = \sum_i  \mu_{\EP f}(O_i') + \mu_{\EP f}(W). $$

We say that we are in the cellular case when the contribution of the cells dominates and in the algebraic case when the contribution of $W$ dominates.  If we are in the cellular case, then there must be $\sim D^n$ cells $O_i'$ so that

\begin{equation} \label{outcelligood} \| \EP f \|_{\BLpkA(B_R)}^p \lesssim D^n \mu_{\EP f} (O_i'). \end{equation}

Next we study $\EP f$ on each of these cells $O_i'$.  We define $f_i$ by

\begin{equation} \label{2deff_i} f_i := \sum_{\theta,v : \Ttv \cap O_i' \not= \emptyset} \ftv. \end{equation}

\noindent Since $\EP \ftv$ is essentially supported on $\Ttv$, we see that $\EP f_i$ is almost equal to $\EP f$ on $O_i'$.
Therefore, $\mu_{\EP f_i} (O_i') \sim \mu_{\EP f} (O_i')$.  Now we study the $\mu_{\EP f_i}(O_i')$ using induction -- since $f_i$ involves fewer wave packets than $f$, it is a simpler object, and so it makes sense to assume by induction that our theorem holds for $f_i$.  This leads to the following bound:

\begin{equation} \label{outinducass} \mu_{\EP f} (O_i') \sim \mu_{\EP f_i} (O_i') \le \| \EP f_i \|_{\BLpkA (B_R)}^p \lesssim \left[ C(\eps) R^\eps \right]^p \| f_i \|_{L^2}^p. \end{equation}

Next we analyze $\| f_i \|_{L^2}$.  By the orthogonality of the $\ftv$, we have

$$ \sum_i \| f_i \|_{L^2}^2 \sim \sum_i \sum_{\theta,v : \Ttv \cap O_i' \not= \emptyset} \| \ftv \|_{L^2}^2
\sim \sum_{\theta, v} | \# \textrm{ cells } O_i' \textrm{ so that } \Ttv \cap O_i' \not= \emptyset | \cdot \| \ftv \|_{L^2}^2. $$

By (\ref{outtubeDcells}), each tube $\Ttv$ enters $\lesssim D$ cells $O_i'$, and so we get

\begin{equation*} \sum_i \| f_i \|_{L^2}^2 \lesssim D \sum_{\theta, v}  | \| \ftv \|_{L^2}^2 \sim D \| f \|_{L^2}^2. \end{equation*}

\noindent Since there are $\sim D^n$ cells $O_i'$ that obey (\ref{outcelligood}), we see that most of them must also obey

\begin{equation} \label{outfil2} \| f_i \|_{L^2}^2 \lesssim D^{1-n} \| f \|_{L^2}^2. \end{equation}

Combining this bound with (\ref{outcelligood}) and our inductive assumption (\ref{outinducass}), we get

$$  \| \EP f \|_{\BLpkA(B_R)}^p \le \left[ C D^{n + (1-n) \frac{p}{2}} \right] [C(\eps) R^\eps]^p \| f \|_{L^2}^p. $$

In this equation, the constant $C$ is the implicit constant from the various $\lesssim$'s.  It does not depend on $D$.  The induction closes as long as the term in brackets is $\le 1$.  Since we can choose the constant $D$, we can arrange that the induction closes as long as the exponent of $D$ is negative.  Given our value of $p$, we can check that the exponent of $D$ is $\le 0$, and the induction closes. 

Now we turn to the algebraic case.  In this case, the measure $\mu_{\EP f}$ is concentrated in $W$ -- the $R^{1/2}$-neighborhood of $Z(P)$ - a degree $D$ algebraic variety of dimension $n-1$.  There are two types of wave packets that contribute to $\mu_{\EP f}$ on $W$, which we describe roughly as follows:

\begin{itemize}

\item {\bf Tangential wave packets}: wave packets that are essentially contained in $W$.  For these wave packets, the direction of the tube $\Ttv$ is (nearly) tangent to $Z = Z(P)$.    

\item {\bf Transverse wave packets}: wave packets that cut across $W$.    

\end{itemize}

We first discuss the case that the tangential wave packets dominate.  To simplify the exposition, let us imagine for the moment that the variety $Z$ is a hyperplane, so $W$ is a planar slab of dimensions $R^{1/2} \times R \times ... \times R$.  We can also imagine that $f$ has the form $f = \sum \ftv$, where all the tubes $\Ttv$ are contained in $W$.

We study this case using induction on the dimension.  In the tangential algebraic case, the behavior of $\EP f$ on the hyperplane $Z$ can be controlled by applying Theorem \ref{broadEP} in dimension $n-1$.  Here is one way to set this up.  There is a standard $L^2$ estimate giving

\begin{equation} \label{outstandL2} \| \EP f \|_{L^2(B_R)} \lesssim R^{1/2} \| f \|_{L^2}. \end{equation}

\noindent It is not hard to reduce Theorem \ref{broadEP} to the case that $\| \EP f \|_{L^2(B_R)} \sim R^{1/2} \| f \|_{L^2}$, and so we can think of Theorem \ref{broadEP} in the equivalent form

$$ \| \EP f \|_{\BLpkA(B_R)} \lesssim R^{-1/2 + \eps} \| \EP f \|_{L^2(B_R)}. $$

\noindent We can apply Theorem \ref{broadEP} in dimension $n-1$ to study $\EP f$ on $Z$, and we get the estimate

$$ \| \EP f \|_{\BL^{\bar p(k, n-1)}_{k, A}(Z \cap B_R)} \lesssim R^{-1/2 + \eps} \| \EP f \|_{L^2(Z \cap B_R)}. $$

Now $\bar p(k,n-1) > \bar p(k,n) = p$.  Interpolating between this estimate and the $L^2$ estimate (\ref{outstandL2}), we get a bound of the form

$$ \| \EP f \|_{\BLpkA(Z \cap B_R)} \lesssim R^{-1/2 + e + \eps} \| \EP f \|_{L^2(Z \cap B_R)}, $$

\noindent for $p = \bar p(k,n)$, where $e$ is an exponent depending on $k, n$.  The same estimate holds not just for $Z$ but for any translate of $Z$ in the slab $W$.  These estimates control the behavior of $\EP f$ in the directions tangent to $Z$.  

To get a good estimate in the tangential algebraic case, we also have to control the behavior of $\EP f$ in the direction transverse to $Z$.  We will show that there is a direction transverse to the hyperplane $Z$ so that $| \EP f(x)|$ is morally constant as we move $x$ in this direction for distance $\lesssim R^{1/2}$.  We call this behavior a transverse equidistribution estimate.  It implies that 

$$ \| \EP f \|_{L^2(B_R)}^2 \sim \| \EP f \|_{L^2(W)}^2 \sim R^{1/2} \| \EP f \|_{L^2(Z \cap B_R)}^2, \textrm{ and}$$

$$\| \EP f \|_{\BLpkA(B_R)}^p \sim \| \EP f \|_{\BLpkA(W)}^p \sim R^{1/2} \| \EP f \|_{\BLpkA(Z \cap B_R)}^p. $$  

Using these estimates, we can control $ \| \EP f \|_{\BLpkA(B_R)}$ by induction on the dimension:

$$ \| \EP f \|_{\BLpkA(B_R)}^p \sim  \| \EP f \|_{\BLpkA(W)}^p \sim R^{1/2}  \| \EP f \|_{\BLpkA(Z \cap B_R)}^p $$

$$ \lesssim R^{1/2} R^{(-1/2 + e + \eps)p} \| \EP f \|_{L^2(Z \cap B_R)}^p \sim R^{1/2} R^{(-1/2 + e + \eps)p} \left( R^{-1/2} \| \EP f \|_{L^2(B_R)}^2 \right)^{p/2}. $$

There are a lot of messy powers of $R$ in this computation.  But plugging in $p = \pkn$ and working out the exponent, one gets $\| \EP f \|_{\BLpkA(B_R)} \lesssim R^{-1/2 + \eps} \| \EP f \|_{L^2(B_R)} \lesssim R^\eps \| f \|_{L^2}$, which closes the induction in the tangential algebraic case.  The exponent $p = \pkn = 2 \cdot \frac{n +  k}{n + k - 2}$ is exactly the exponent needed to make the powers of $R$ work out in this computation.  

Next we sketch the reason for this transverse equidistribution.  From the definition of $\EP$, we see that $\EP \ftv$ has Fourier transform supported in a ball of radius $R^{-1/2}$ around

$$ \xi(\theta) = (\omega_{\theta,1}, ..., \omega_{\theta, n-1}, | \omega_\theta|^2),$$

\noindent where $\omega_\theta$ denotes the center of $\theta$.  We know that all the wave packets $\EP \ftv$ in $\EP f$ are supported in tubes $\Ttv \subset W$.  This situation restricts the directions $G(\theta)$, which in turn restricts the frequencies $\xi(\theta)$.
The connection between $G(\theta)$ and $\xi(\theta)$ is simplest for a slightly different operator - the extension operator for the sphere.  In that case, $G(\theta) = \xi(\theta)$.  Since the tubes $\Ttv$ all lie in $W$, the directions $G(\theta)$ all lie in the $R^{-1/2}$-neighborhood of the subspace $V \subset \RR^n$ parallel to $Z$.  Therefore, the frequencies $\xi(\theta)$ all lie in the $R^{-1/2}$-neighborhood of $V$ also.  
So the Fourier transform of $\EP f$ is supported in the $R^{-1/2}$-neighborhood of $V$.  If $\ell$ denotes a line perpendicular to $V$ (or to $Z$), then the restriction of $\EP f$ to $\ell$ has Fourier transform supported in a ball of radius $R^{-1/2}$.  Therefore, $| \EP f |$ is morally constant as we move along the line $\ell$ for distances $\lesssim R^{1/2}$.

For the extension operator for the paraboloid, the situation is similar but a touch messier.  We know that the directions $G(\theta)$ all lie in the $R^{-1/2}$-neighborhood of the plane $V$.  A short calculation shows that the frequencies $\xi(\theta)$ all lie in the $R^{-1/2}$-neighborhood of an affine hyperplane $V'$.  The hyperplane $V'$ is not equal to $V$, but the angle between $V$ and $V'$ is fairly small.  If $\ell$ is perpendicular to the plane $V'$, then it still follows that $|\EP f|$ is morally constant as we move along $\ell$ for distances $\lesssim R^{-1/2}$.  The line $\ell$ is no longer exactly perpendicular to the original plane $Z$, but it is still quantitatively transverse to $Z$, and this is good enough for our application.   

In this sketch, we assumed that $Z$ is a hyperplane.  But in the real proof we cannot assume this.  We have to set up the induction on dimension in a different way, taking into account the possibility that $Z$ is curved.  We explain this in the next subsection.

Finally, it can happen that the transverse wave packets dominate.  In this last case, $\mu_{\EP f}(W)$ dominates $\mu_{\EP f}(B_R)$, but the wave packets transverse to $W$ make the main contribution to $\mu_{\EP f}(W)$.  In this case, we can imagine that $f = \sum \ftv$ where each tube $\Ttv$ is transverse to $W$.  Recall that the number of times a line can cross the hypersurface $Z$ is $\le D \lesssim 1$.  Similarly, we will prove that the number of times a tube $\Ttv$ can cross the surface $W$ is $\lesssim 1$.

In this case, we subdivide the ball $B_R$ into smaller balls $B_j$ with radius $\rho \ll R$.  A tube $\Ttv$ enters $\sim (R /\rho) \gg 1$ of these balls.  But because of the discussion above, a tube $\Ttv$ can cross $W$ transversely in $\lesssim 1$ balls $B_j$.  We define

$$ \TTT_{j, trans} := \{ (\tv) | \Ttv \textrm{ crosses $W$ transversely in } B_j \}, $$

$$ f_{j, trans} := \sum_{(\tv) \in \TTT_{j, trans}} \ftv. $$

Since a tube $\Ttv$ can cross $W$ transversely in $\lesssim 1$ balls $B_j$, each $(\tv)$ belongs to $\lesssim 1$ sets $\TTT_{j, trans}$.  Therefore,

\begin{equation} \label{outfjtransl2} \sum_j \| f_{j, trans} \|_{L^2}^2 \lesssim \| f \|_{L^2}^2. \end{equation}

Since we assumed that all the wave packets in $f$ intersect $W$ transversely, $\EP f_{j, trans}$ is essentially equal to $\EP f$ on $W \cap B_j$.  Therefore,

$$ \mu_{\EP f} (W \cap B_j) \le \mu_{\EP f_{j, trans}} (B_j). $$

Since we are in the algebraic case, and since we assumed that all the tubes $\Ttv$ intersect $W$ transversely, we have

\begin{equation} \label{outtotalsumj} \mu_{\EP f} (B_R) \lesssim \mu_{\EP f} (W) \sim \sum_j \mu_{\EP f}(W \cap B_j) \lesssim \sum_j \mu_{\EP f_{j,trans}} (B_j). \end{equation}

By induction on the radius, we can assume that

$$ \mu_{\EP f_j}(B_j) \le \left[ C(\eps) \rho^\eps \right]^p \| f_{j, trans} \|_{L^2}^p. $$

Plugging this bound into (\ref{outtotalsumj}) and then applying (\ref{outfjtransl2}), we get

$$ \| \EP f \|_{\BLpkA(B_R)}^p \lesssim \left[ C(\eps) \rho^\eps \right]^p \sum_j \| f_{j, trans} \|_{L^2}^p.  $$

Since $p > 2$, we get $\sum_j \| f_{j, trans} \|_{L^2}^p \le ( \sum_j \| f_{j, trans} \|_{L^2}^2)^{p/2} \lesssim \| f \|_{L^2}^p$.  Therefore, we have 
 
$$ \| \EP f \|_{\BLpkA(B_R)}^p \lesssim  \left[ C(\eps) \rho^\eps \right]^p \| f \|_{L^2}^p. $$

Since $(R/ \rho)$ is large, $C \rho^{\eps} < R^\eps$, and so this closes the induction in the transverse algebraic case.

\subsection{Studying wave packets tangent to a variety} In the preceding sketch, we considered the special case that the variety $Z$ is a hyperplane.  In this special situation, the tangential algebraic case reduces to the original theorem in dimension $n-1$.  In the full proof we need to consider curved varieties $Z$, and so we have to do the induction on the dimension in a different way.

If $Z$ is an $m$-dimensional variety in $\RR^n$, then we say that the tube $\Ttv$ is $\alpha$-tangent to $Z$ in $B_R$ if the following two conditions hold:

\begin{itemize}

\item Distance condition:

$$ \Ttv \subset N_{\alpha R}(Z) \cap B_R. $$

\item Angle condition: If $z \in Z \cap N_{\alpha R}(\Ttv)$, then

$$ \Angle (T_z Z, G(\theta) ) \le \alpha. $$

\end{itemize}

For each dimension $m$, we will choose an angle $\alpha_m$ slightly larger than $R^{-1/2}$, and then we define

$$ \TTT_{Z} := \{ (\tv): \Ttv \textrm{ is $\alpha_m$-tangent to $Z$ in $B_R$} \}. $$

Our main technical result is Proposition \ref{mainind}.  It says that if $Z$ is an $m$-dimensional variety of controlled degree, and if $f = \sum_{(\tv) \in \TTT_{Z}} \ftv$, then (for a range of exponents $p \ge 2$),

$$\| \EP f \|_{\BLpkA(B_R)} \le M \| f \|_{L^2}, $$ 

\noindent where $M$ is a fairly complicated expression which depends on the parameters of the setup, including the exponent $p$, the radius $R$, the dimension $m$, the value of $A$, etc.  

The proof of Proposition \ref{mainind} is by induction on the dimension $m$.  The base of the induction is the case $m=k-1$.  In this case, $\| \EP f \|_{\BLpkA(B_R)}$ is negligibly small.  To see this, consider a small ball $B_{K^2}$.  The function $\EP f$ is essentially supported in $N_{R^{1/2}}(Z) \cap B_R$, so we can assume $B_{K^2} \subset N_{R^{1/2}}(Z) \cap B_R$.  Because of the angle condition, all the wave packets $\Ttv \in \TTT_Z$ that pass through $B_{K^2}$ have direction $G(\theta)$ within a small angle of an $(k-1)$-plane -- the plane $T_z Z$ for a point $z \in Z$ near to $B_{K^2}$ -- and so the ball $B_{K^2}$ makes a negligible contribution to $\| \EP f \|_{\BLpkA}$.  This provides the base of the induction.

The proof of Proposition \ref{mainind} follows the rough outline of our first sketch.  We again use polynomial partitioning.  Under the hypotheses of the proposition, we know that $\EP f$ is essentially supported in $N_{R^{1/2}}(Z)$ for a given $m$-dimensional variety $Z$.  We want to find a polynomial $P$ so that $Z(P)$ cuts $N_{R^{1/2}}(Z)$ into smaller cells.  To do this, we choose new orthogonal coordinates $y_1, ..., y_n$ so that the projection of $Z$ to the $(y_1, ..., y_m)$-plane is non-degenerate on a significant portion of $Z$.  Then we let $P$ be a polynomial in $y_1, ..., y_m$.  In this case, $Z(P)$ intersects $Z$ transversely (at least on a significant portion of $Z$), and this makes the polynomial partitioning work.  Essentially everything works as in the first sketch, except that because $P$ depends on only $m$ variables, the number of cells $O_i$ is only $\sim D^m$.  This will affect the final exponents, but the method of the argument is the same.  If $\mu_{\EP f}$ is concentrated on the cells $O_i'$, then we can prove the desired bounds by induction.  Otherwise, $\mu_{\EP f}$ is concentrated in the $R^{1/2}$-neighborhood of a lower-dimensional variety $Y = Z \cap Z(P)$.  

Now we turn to this algebraic case.  In the algebraic case, there is a lower-dimensional variety $Y$, with $\Deg Y \lesssim \Deg Z$, so that 

$$\mu_{\EP f} ( N_{R^{1/2}}(Y) \cap B_R) \gtrsim \mu_{\EP f} (B_R) \sim \mu_{\EP f} (N_{R^{1/2}}(Z) \cap B_R). $$

\noindent There are two types of wave packets that contribute to $\mu_{\EP f}$ on $N_{R^{1/2}}(Y) \cap B_R$: tangential wave packets, which lie in $N_{R^{1/2}}(Y)$ and run tangent to $Y$, and transverse wave packets, which cut across $N_{R^{1/2}}(Y)$.  (Recall that all the wave packets are tangent to $Z$ by hypothesis.)

In our current setup, if all the wave packets are tangent to $Y$, then we get the desired estimate just by induction on the dimension $m$.

However, there may be a mix of transverse and tangential wave packets.  If we let $f_{tang}$ be the sum of the tangential wave packets and $f_{trans}$ be the sum of the transverse wave packets, then the quasi-triangle inequality for $\BLpkA$ (\ref{introquasitriangle}) gives

$$ \| \EP f \|_{\BLpkA(B_R)} \lesssim \| \EP f_{tang} \|_{\BLpkAt(B_R)} +  \| \EP f_{trans} \|_{\BLpkAt(B_R)}. $$

\noindent This is the step in the proof where we need to use the quasi-triangle inequality. 

We can handle the tangential terms by induction on the dimension, and now we turn to the transverse terms.  As in the last sketch, we decompose $B_R = \cup_j B_j$, where each ball $B_j$ has radius $\rho \ll R$.  We define $f_{j, trans}$ as above to be the sum of wave packets that intersect $N_{R^{1/2}}(Y)$ transversely in $B_j$.  As in the previous sketch, geometric arguments show that a tube $\Ttv$ can intersect $N_{R^{1/2}}(Y)$ transversely in $\lesssim 1$ balls $B_j$, and so

$$\sum_j \| f_{j, trans} \|_{L^2}^2 \lesssim \| f \|_{L^2}^2. $$

Next we want to study $\| \EP f_{j, trans} \|_{\BLpkAt(B_j)}$ by using induction on the radius.  This step is more complicated than in the previous sketch.  We know that $f$ is concentrated on wave packets that are tangent to $Z$ on $B_R$, and we need to use that information.  We expand $\EP f_{j, trans}$ into wave packets on the ball $B_j$.  Since $B_j$ has radius $\rho$, each wave packet is essentially supported on a tube of radius $\rho^{1/2}$ and length $\rho$ in $B_j$.  When we examine this wave packet decomposition, it is not exactly true that all the wave packets are tangent to $Z$ in $B_j$ -- in fact, something better is true.  The wave packets of $f_{j, trans}$ on $B_j$ all lie in $N_{R^{1/2}}(Z) \cap B_j$, but they don't necessarily lie in $N_{\rho^{1/2}}(Z) \cap B_j$.  We cover $N_{R^{1/2}}(Z) \cap B_j$ with disjoint translates of $N_{\rho^{1/2}}(Z) \cap B_j$:

$$ N_{R^{1/2}}(Z) \cap B_j = \cup_b N_{\rho^{1/2}}(Z + b) \cap B_j. $$

\noindent Now it turns out that each wave packet of $f_{j, trans}$ lies in one of these translates.  We let $f_{j, trans, b}$ be the sum of the wave packets that lie in $N_{\rho^{1/2}(Z+b)} \cap B_j$.  Now we have $f_{j, trans} = \sum_b f_{j, trans, b}$, and

$$ \| \EP f_{j, trans} \|_{\BLpkAt(B_j)}^p \sim \sum_b \| \EP f_{j, trans} \|_{\BLpkAt(N_{\rho^{1/2}}(Z+b) \cap B_j)}^p \sim \sum_b \| \EP f_{j, trans, b} \|_{\BLpkAt(B_j)}^p. $$

The wave packets of $\EP f_{j, trans, b}$ are tangent to the $m$-dimensional variety $Z+b$ on $B_j$.  Therefore, we can study $ \| \EP f_{j, trans, b} \|_{\BLpkAt(B_j)}$ by induction on the radius: we can assume that

$$ \| \EP f_{j, trans, b} \|_{\BLpkAt(B_j)} \le M(\rho, A/2) \| f_{j, trans, b} \|_{L^2}. $$

(Here we write $M(\rho, A/2)$ because $M$ depends on the radius (which is $\rho$) and because we have $A/2$ in place of $A$.)  Putting together what we have learned so far in the transverse algebraic case, we have 

$$ \| \EP f \|_{\BLpkA(B_R)}^p \lesssim \sum_j \| \EP f_{j, trans} \|_{\BLpkAt(B_j)}^p $$

$$ \sim \sum_{j, b} \| \EP f_{j, trans, b} \|_{\BLpkAt(B_j)}^p \le M(\rho, A/2)^p \sum_{j,b} \| f_{j, trans, b} \|_{L^2}^p. $$

To get our final bound, it remains to control $\sum_{j,b} \| f_{j, trans, b} \|_{L^2}^p$.  Since each wave packet of $f_{j, trans}$ lies in exactly one $f_{j, trans,b}$, it is easy to check that $\| f_{j, trans} \|_{L^2}^2 \sim \sum_b \| f_{j, trans, b} \|_{L^2}^2$.  We also already know that $\sum_j \| f_{j, trans} \|_{L^2}^2 \lesssim \| f \|_{L^2}^2$.  And so we see that

$$ \sum_{j,b} \| f_{j, trans,b} \|_{L^2}^p \lesssim \left( \max_{j,b} \| f_{j, trans, b} \|_{L^2}^{p-2} \right) \| f \|_{L^2}^2. $$

The last ingredient of the proof is an estimate for $\max_{j,b} \| f_{j, trans, b} \|_{L^2}$, which has to do with transverse equidistribution.  Recall that since $f$ is concentrated on wave packets in $\TTT_Z$, $\EP f$ is equidistributed in directions transverse to $Z$.  Recall that $N_{R^{1/2}}(Z) \cap B_j$ is covered by thinner neighborhoods $N_{\rho^{1/2}}(Z+b) \cap B_j$.  The number of these thinner neighborhoods is $\sim\left( \frac{R^{1/2}}{\rho^{1/2}}\right)^{n-m}$.   Because of transverse equidistribution, each of these neighborhoods receives an even share of the $L^2$ norm of $\EP f_{j, trans}$.  In other words, for each $b$,

$$ \| \EP f_{j, trans} \|^2_{L^2(N_{\rho^{1/2}}(Z + b) \cap B_j)} \lesssim \left( \frac{R^{1/2}}{\rho^{1/2}}\right)^{-(n-m)}  \| \EP f_{j, trans} \|^2_{L^2(N_{R^{1/2}}(Z) \cap B_j)}. $$

Using this inequality, we prove the desired estimate for $\| f_{j, trans, b} \|_{L^2}$:

\begin{equation} \label{outequi} \| f_{j, trans, b} \|_{L^2}^2 \lesssim \left( \frac{R^{1/2}}{\rho^{1/2}}\right)^{-(n-m)} \| f_{j, trans} \|_{L^2}^2 \lesssim \left( \frac{R^{1/2}}{\rho^{1/2}}\right)^{-(n-m)} \| f \|_{L^2}^2 . \end{equation}

\subsection{Outline of the paper}
We carry out the proof of Theorem \ref{broadEP} over Sections 3 - 8 of the paper.  Section 3 reviews some standard facts about wave packets.  Section 4 proves some basic properties of the broad ``norms'' $\BLpkA$.  Section 5 contains tools from algebraic geometry (and differential geometry) that we will use to study the geometry of the algebraic varieties that appear in the proofs.  
Section 6 proves the first transverse equidistribution estimate.  For any ball $B(y, \rho) \subset \RR^n$, there is a wave packet decomposition for $\EP f$ on the ball $B(y, \rho)$.   Section 7 is about the relationship between the original wave packet decomposition on the ball $B_R$ and the wave packet decomposition adapted to a smaller ball $B(y, \rho) \subset B_R$.  This lets us state and prove a second version of the transverse equidistribution estimate, corresponding to (\ref{outequi}) above.  
With the background and tools from these sections, we prove Theorem \ref{broadEP} in Section 8.

Following \cite{BG}, Section 9 explains how $k$-broad estimates imply regular $L^p$ estimates of the form $\| \EP f \|_{L^p} \lesssim \| f \|_{L^p}$.  This argument finishes the proof of Theorem \ref{L^pEP}.  Section 10 is an appendix which helps to keep track of the parameters.  Unfortunately, there are quite a few parameters in the paper -- various $\delta$'s, $R$, $K$, $A$, etc.  The appendix lists all the parameters and how they relate to each other.  
Section 11 discusses further directions and open problems.  

\section{Basic setup with wave packets} \label{basicwavepack}

Let $f$ be a function on $B^{n-1}$.  We first break up $f$ into pieces $\ftv$ that are localized in both position and frequency.

Cover $B^{n-1}$ by finitely overlapping balls $\theta$ of radius $R^{-1/2}$.  Let $\psi_\theta$ be a smooth partition of unity adapted to this cover, and write $f = \sum_\theta \psi_\theta f$.

Next we break up $\psi_\theta f$ according to frequency.  Cover $\RR^{n-1}$ by finitely overlapping balls of radius $\sim R^{\frac{1 + \dt}{2}}$, centered at vectors $v \in R^{\frac{1 + \dt}{2}} \ZZ^{n-1}$.  Let $\eta_v$ be a smooth partition of unity adapted to this cover.  We can now write

$$ f = \sum_{\tv} \left( \eta_v (\psi_\theta f)^\wedge \right)^{\vee} = \sum_{\tv} \eta_v^\vee * (\psi_\theta f). $$

Note that $\eta_v^\vee(x)$ is rapidly decaying for $|x| \gtrsim R^{\frac{1 - \dt}{2}}$.  Choose smooth functions $\tilde \psi_\theta$ so that $\tilde \psi_\theta$ is supported on $\theta$, but $\tilde \psi_\theta$ is 1 on a  small neighborhood of the support of $\psi_\theta$.  A bit more precisely, we would like $\tilde \psi_\theta = 1$ on a $c R^{-1/2}$ neighborhood of the support of $\psi_\theta$ for a small constant $c > 0$.  Now we define

$$ \ftv := \tilde \psi_\theta \left[ \eta_v^\vee * (\psi_\theta f) \right]. $$

Because of the rapid decay of $\eta_v^\vee$,

$$ \left|| \ftv - \eta_v^\vee * (\psi_\theta f) \right\|_{L^\infty} \le \RD(R) \| f \|_{L^2}. $$

Therefore, we see that

$$ f = \sum_\tv \ftv + \Err, \textrm{ where } \| \Err \|_{L^\infty} \le \RD(R) \| f \|_{L^2}. $$

Terms of the form $\RD(R) \| f \|_{L^2}$ are negligibly small in terms of all of our estimates.  These rapidly decaying errors will occur from time to time during our arguments.

The functions $\ftv$ are approximately orthogonal.  For any set $\TTT$ of pairs $(\tv)$, we have

\begin{equation} \label{orthog} \left\| \sum_{(\tv) \in \TTT} \ftv \right\|_{L^2}^2 \sim \sum_{(\tv) \in \TTT} \| \ftv \|_{L^2}^2. \end{equation}

The decomposition $f = \sum_{\tv} \ftv$ is useful in this problem because $\EP \ftv$ is localized in space.  
For each $(\tv)$, there is a corresponding tube $\Ttv$, where $\EP \ftv$ is essentially supported.  Let $\omega_\theta$ denote the center of $\theta$.  We define $\Ttv$ by

\begin{equation} \label{defTtv}
\Ttv := \{ (x', x_n) \in B_R \textrm{ so that } |x' + 2 x_n \omega_\theta + v | \le R^{1/2 + \dt} \}.
\end{equation}

\begin{lemma} \label{wavepacket} If $x \in B_R \setminus \Ttv$, then
$$ | \EP \ftv (x) | \le \RD(R) \| f \|_{L^2}. $$
\end{lemma}

We sketch the proof by stationary phase.  We note that $( e^{- i v \omega} \ftv )$ has Fourier transform essentially supported in $B_{R^{1/2 + \dt/2}}$.  Therefore, we have

$$ \max_\theta | e^{-i v \omega} \ftv | \lesssim R^{(n-1) \frac{\dt}{2} } \Avg_\theta |e^{-i v \omega} \ftv|. $$

Moreover, taking derivatives, we see that

$$ \max_\theta \left| \partial_\omega^k \left( e^{- i v \omega} \ftv \right) \right| \lesssim \left( R^{1/2 + \dt/2} \right)^k R^{(n-1) \frac{\dt}{2} } \Avg_\theta |e^{-i v \omega} \ftv|. $$

Let $\eta_\theta$ be a smooth bump which is equal to 1 on $\theta$.  Since $\theta$ has diameter $R^{-1/2}$, we can estimate the derivatives $| \partial^k \eta_\theta | \lesssim R^{k/2}$.  Now we write $\EP \ftv (x)$ as

\begin{equation} \label{Eftvintparts} \EP \ftv (x) = \int \left[ \eta_\theta e^{i ( x'  \omega + x_n | \omega |^2 + v \omega )} \right] \cdot \left( e^{- i v \omega} \ftv \right) . \end{equation}

We let $\Psi(\omega) := x' \omega + x_n | \omega|^2 - v \omega$.  We note that

$$ \partial_\omega \Psi = x' + 2 x_n \omega + v. $$

If $x \in B_R \setminus \Ttv$, then $| x' + 2 x_n \omega_\theta + v|  = | \partial_\omega \Psi(\omega_\theta)| > R^{1/2 + \dt}$.  We know that $|x_n| \le R$, and so for any $\omega \in \theta$, $| 2 x_n \omega - 2 x_n \omega_\theta| \lesssim R^{1/2}$.  Therefore, for any $\omega \in \theta$,  

$$| \partial_\omega \Psi(\omega) | \gtrsim R^{1/2 + \dt}.$$

By applying integration by parts to (\ref{Eftvintparts}) many times, we see that $| \EP \ftv (x) | \le \RD(R) \| \ftv \|_{L^2}$ as desired.

The tube $\Ttv$ is a cylinder of length $R$ and radius $\sim R^{1/2 + \dt}$.  It points in the direction $G(\omega_\theta)$, where $G(\omega)$ is the unit vector given by

$$ G(\omega) = \frac{ (-2 \omega_1, ..., -2 \omega_{n-1}, 1) }{| (-2 \omega_1, ..., -2 \omega_{n-1}, 1) |}. $$

For each $\omega \in B^{n-1}$, we also define a frequency $\xi(\omega)$.  Based on the formula for $\EP f$, the frequency $\xi(\omega)$ is given by

$$ \xi(\omega) := (\omega_1, ..., \omega_{n-1}, |\omega|^2). $$

We let $\xi(\theta)$ denote the image of $\theta$ under $\xi$:

$$ \xi(\theta) := \{ (\omega_1, ..., \omega_{n-1}, |\omega|^2) | \omega \in \theta \}. $$

In a distributional sense, the Fourier transform of $\EP \ftv$ is supported in $\xi(\theta)$.  Also, if $\eta_R$ denotes a smooth bump on $B_R$ (of height 1), then the Fourier transform of $\eta_R \EP \ftv$ is essentially supported in $N_{R^{-1}} ( \xi(\theta) )$.  

We introduce a little notation.  If $\TTT_\alpha$ is any set of pairs $(\tv)$, then we say that $f$ is concentrated on wave packets from $\TTT_\alpha$ if

$$ f = \sum_{(\tv) \in \TTT_\alpha} \ftv + \RD(R) \| f \|_{L^2}. $$

\noindent Also, for any $f$, and for any set $\TTT_\alpha$, we define

$$ f_\alpha = \sum_{(\tv) \in \TTT_\alpha} \ftv. $$

\subsection{Orthogonality}

For any fixed $x_n$, $\EP f$ restricted to $\RR^{n-1} \times \{ x_n \}$ can be described as an inverse Fourier transform:

$$ \EP f (x_1, ..., x_{n-1}, x_n) = \left(e^{i x_n |\omega|^2} f(\omega) \right)^{\vee} (x_1, ..., x_{n-1}). $$

Applying Plancherel, we get

\begin{equation} \label{orthog1}
\| \EP f \|_{L^2(\RR^{n-1} \times \{ x_n \} ) } = \| f \|_{L^2}.
\end{equation}

We record a couple of simple corollaries of this statement.

\begin{lemma} \label{basicL2} $$ \| \EP f \|_{L^2(B_R)} \lesssim R^{1/2} \| f \|_{L^2}. $$
\end{lemma}

\begin{proof}

$$ \int_{B_R} | \EP f |^2 \le \int_{- R}^R \left( \int_{B^{n-1}(R)} | \EP f |^2 dx' \right) dx_n \le 2 R \| f \|_{L^2}^2. $$

\end{proof}

\begin{lemma} \label{orthogloc} Suppose that $f$ is concentrated on a set of wave packets $\TTT$ and that for every $(\tv) \in \TTT$, $\Ttv \cap (\RR^{n-1} \times \{ x_n \}) \subset B^{n-1}(z_0, r) \times \{ x_n \}$.  Then
\begin{equation} \label{orthog2} \| \EP f \|_{L^2(B^{n-1}(z_0, r) \times \{ x_n \} ) } = \| f \|_{L^2} + \RD(R) \| f \|_{L^2}.
\end{equation}
\end{lemma}

\begin{lemma} \label{orthogloc2} Suppose that $f$ is concentrated on a set of wave packets $\TTT$ and that for every  
$(\tv) \in \TTT$, $\Ttv \cap B(z, r) \not= \emptyset$, for some radius $r \ge R^{1/2 + \dt}$ .  Then

\begin{equation} \label{orthog3} \| \EP f \|^2_{L^2(B(z, 10 r )) } \sim r \| f \|_{L^2}^2.
\end{equation}
\end{lemma}

\begin{proof} For each $x_n$ in the range $z_n - r \le x_n \le z_n + r$, and for each $(\tv) \in \TTT$, the intersection $\Ttv \cap \RR^{n-1} \times \{ x_n \}$ is contained in $B(z, 5 r)$.  By the last lemma, we see that

\begin{equation*} 
\| \EP f \|_{L^2(B(z, 5 r ) \cap \RR^{n-1} \times \{ x_n \} ) } = \| f \|_{L^2} + \RD(R) \| f \|_{L^2}.
\end{equation*}

Applying Fubini, we get the desired bound.   \end{proof}

\section{Properties of the broad ``norms'' $\BLpkA$}

We recall the definition of the $k$-broad ``norm'' $\BLpkA$.  Although $\BLpkA$ is not literally a norm, it obeys a version of the triangle inequality and a version of Holder's inequality.  These nice algebraic features helped to motivate this particular definition.

Let $B^{n-1}$ be a disjoint union of (approximate) balls $\tau$ of radius $K^{-1}$ .   For each $\tau$, we define $G(\tau)$ to be the image of $\tau$ under the direction map $G$.  If $\omega_\tau$ is the center of $\tau$, then $G(\tau)$ is essentially a ball of radius $K^{-1}$ around $G(\omega_\tau)$ in $S^{n-1}$.  If $V \subset \RR^n$ is a subspace, then we write $\Angle( G(\tau), V)$ for the smallest angle between any non-zero vectors $v \in V$ and $v' \in G(\tau)$.

For any ball $B_{K^2}$ of radius $K^2$ in $B_R$, we define $\mu_{\EP f}$ as in (\ref{defmuEf}).  

\begin{equation*} \mu_{\EP f} (B_{K^2}) := \min_{V_1, ..., V_A \textrm{ $(k-1)$-subspaces of } \RR^n} \left( \max_{\tau: \Angle(G(\tau), V_a) > K^{-1} \textrm{ for all } a } \int_{B_{K^2}} | \EP f_\tau|^p \right). \end{equation*}

Because this expression is a little long, we abbreviate it as

\begin{equation*} \mu_{\EP f} (B_{K^2}) := \min_{V_1, ..., V_A} \left( \max_{\tau \notin V_a} \int_{B_{K^2}} | \EP f_\tau|^p \right). \end{equation*}

We remark that it is convenient to allow $A = 0$.  If $A = 0$, then we have simply $\mu_{\EP f} (B_{K^2}) = \max_{\tau} \int_{B_{K^2}} | \EP f_\tau|^p$.  

If $U \subset B_R$ is a finite union of balls $B_{K^2}$, then we define $\| \EP f \|_{L^p(U)}$ by:

\begin{equation} \label{defBroadL^pU}
\| \EP f \|_{\BLpkA(U)}^p := \sum_{B_{K^2} \subset U} \mu_{\EP f} (B_{K^2}) . 
\end{equation}

The $k$-broad ``norm'' obeys a weak version of the triangle inequality.

\begin{lemma} \label{triangleineq} Suppose that $f = g+h$ and suppose that $A = A_1 + A_2$, where $A, A_i$ are non-negative integers.  Then

$$ \| \EP f \|_{\BLpkA(U)} \lesssim \| \EP g \|_{\BLp_{k, A_1} (U)} + \| \EP h \|_{\BLp_{k, A_2} (U)}. $$

\end{lemma}

\begin{proof} We expand

$$ \| \EP f \|_{\BLpkA(U)}^p = \sum_{B_{K^2} \subset U} \min_{V_1, ..., V_A} \left( \max_{\tau \notin V_a } \int_{B_{K^2}} | \EP f_\tau|^p \right).$$

Now for each ball $B_{K^2} \subset U$, we have 

$$ \min_{V_1, ..., V_A} \left( \max_{\tau \notin V_a } \int_{B_{K^2}} | \EP f_\tau|^p \right) \lesssim \min_{V_1, ..., V_A} \left( \max_{\tau \notin V_a } \left[ \int_{B_{K^2}} | \EP g_\tau|^p + \int_{B_{K^2}} | \EP h_\tau|^p \right] \right) $$

$$\le \min_{V_1, ..., V_{A_1}} \left( \max_{\tau \notin V_a, 1 \le a \le A_1 } \int_{B_{K^2}} | \EP g_\tau|^p \right) + \min_{V_{A_1 + 1}, ..., V_{A}} \left( \max_{\tau \notin V_a, A_1 + 1 \le a \le A }  \int_{B_{K^2}} | \EP h_\tau|^p \right). $$

Summing over all $B_{K^2} \subset U$, we get

$$  \| \EP f \|_{\BLpkA(U)}^p \lesssim  \| \EP g \|_{\BLp_{k,A_1} (U)}^p  +  \| \EP h \|_{\BLp_{k, A_2} (U)}^p. $$

\end{proof}

The reason that we need a large value of $A$ in Theorem \ref{broadEP} is that we will need to use this triangle inequality many times.  If $A=1$, $\BLp_{k,1} $ does not obey a good triangle inequality.  But if we start with $A$ a large constant, we can use Lemma \ref{triangleineq} many times.  In effect, $\BLpkA$ behaves like a norm as long as we only use the triangle inequality $O_\eps(1)$ times in our argument, and as long as we choose $A = A(\eps)$ large enough.

$\BLpkA$ also obeys a version of (a corollary of) Holder's inequality.  

\begin{lemma} \label{broadholder} Suppose that $1 \le p, p_1, p_2 < \infty$ and $0 \le \alpha_1, \alpha_2 \le 1$ obey $\alpha_1 + \alpha_2 = 1$ and 

$$ \frac{1}{p} = \alpha_1 \frac{1}{p_1} + \alpha_2 \frac{1}{p_2}. $$

Also suppose that $A = A_1 + A_2$.  Then

$$ \| \EP f \|_{\BLpkA(U)} \le \| \EP f \|_{\BL^{p_1}_{k, A_1} (U)}^{\alpha_1} \| \EP f \|_{\BL^{p_2}_{k, A_2} (U)}^{\alpha_2}. $$

\end{lemma}

\begin{proof} The left-hand side is

$$ \left[ \sum_{B_{K^2} \subset U} \min_{V_1, ... V_A}  \max_{\tau \notin V_a} \int_{B_{K^2}} | \EP f_\tau|^p    \right]^{\frac{1}{p}}. $$

Applying the regular Holder inequality to the inner integral, this expression is

$$\le \left[ \sum_{B_{K^2} \subset U} \min_{V_1, ... V_A}  \max_{\tau \notin V_a} \left( \int_{B_{K^2}} | \EP f_\tau|^{p_1} \right)^{\alpha_1 \frac{p}{p_1}} \left( \int_{B_{K^2}} | \EP f_\tau|^{p_2} \right)^{\alpha_2 \frac{p}{p_2}}     \right]^{\frac{1}{p}}. $$

We can bring the maximum over $\tau$ inside, so the last expression is

$$ \le  \left[ \sum_{B_{K^2} \subset U} \min_{V_1, ... V_A} \left(  \max_{\tau \notin V_a}  \int_{B_{K^2}} | \EP f_\tau|^{p_1} \right)^{\alpha_1 \frac{p}{p_1}} \left( \max_{\tau \notin V_a}  \int_{B_{K^2}} | \EP f_\tau|^{p_2} \right)^{\alpha_2 \frac{p}{p_2}}     \right]^{\frac{1}{p}}. $$

Now we cannot bring the minimum inside the parentheses.  But we can split $V_1, ..., V_A$ into $V_1, ..., V_{A_1}$ and $V_{A_1 + 1}, ..., V_{A}$.  If we weaken the first condition $\tau \notin V_1 ... V_A$ to $\tau \notin V_1, ..., V_{A_1}$, and if we weaken the second condition $\tau \notin V_1, ..., V_A$ to $\tau \notin V_{A_1 + 1}, ..., V_A$, then we see that the last expression is bounded by

$$ \le  \left[ \sum_{B_{K^2} \subset U} \left( \min_{V_1, ... V_{A_1}} \max_{\tau \notin V_a, 1 \le a \le A_1}  \int_{B_{K^2}} | \EP f_\tau|^{p_1} \right)^{\alpha_1 \frac{p}{p_1}} \left(\min_{V_{A_1+1}, ... V_A} \max_{\tau \notin V_a, A_1+1 \le a \le A}  \int_{B_{K^2}} | \EP f_\tau|^{p_2} \right)^{\alpha_2 \frac{p}{p_2}}     \right]^{\frac{1}{p}}. $$

Now we apply Holder to the initial sum over $B_{K^2} \subset U$, and we get

$$ \le \left[   \sum_{B_{K^2} \subset U} \min_{V_1, ... V_{A_1}} \max_{\tau \notin V_a, 1 \le a \le A_1}  \int_{B_{K^2}} | \EP f_\tau|^{p_1}  \right]^{\frac{\alpha_1}{p_1}} \left[   \sum_{B_{K^2} \subset U} \min_{V_{A_1 + 1}, ... V_{A}} \max_{\tau \notin V_a, A_1 + 1 \le a \le A}  \int_{B_{K^2}} | \EP f_\tau|^{p_2}  \right]^{\frac{\alpha_2}{p_2}}$$

$$ = \| \EP f \|_{\BL^{p_1}_{k, A_1} (U)}^{\alpha_1} \| \EP f \|_{\BL^{p_2}_{k, A_2} (U)}^{\alpha_2}. $$

\end{proof}

\section{Tools from algebraic geometry}

\subsection{Transverse complete intersections} Over the course of our argument we will work not just with algebraic hypersurfaces but algebraic varieties of all dimensions.  We write $Z(P_1, ..., P_{n-m})$ for the set of common zeroes of the polynomials $P_1, ..., P_{n-m}$.  Throughout the paper, we will work with a nice class of varieties called transverse complete intersections.  The variety $Z(P_1, ..., P_{n-m})$ is a transverse complete intersection if

\begin{equation} \label{deftci} \nabla P_1(x) \wedge ... \wedge \nabla P_{n-m}(x) \not= 0 \textrm{ for all } x \in Z(P_1, ..., P_{n-m}). \end{equation}

By the implicit function theorem, a transverse complete intersection $Z(P_1, ..., P_{n-m})$ is a smooth $m$-dimensional manifold.  Because of Sard's theorem, there are lots of transverse complete intersections.  Here is a lemma making this precise.

\begin{lemma} \label{sardtci} If $P$ is a polynomial on $\RR^n$, then for almost every $c_0 \in \RR$, $Z( P + c_0)$ is a transverse complete intersection.

More generally, suppose that $Z(P_1, ..., P_{n-m})$ is a transverse complete intersection and that $P$ is another polynomial.  Then for almost every $c_0 \in \RR$, $Z(P_1, ..., P_{n-m}, P + c_0)$ is a transverse complete intersection.
\end{lemma}

\begin{proof} We begin with the first case.  We know that $P: \RR^n \rightarrow \RR$ is a smooth function, and so by Sard's theorem, almost every $y \in \RR$ is a regular value for $P$.  But if $- c_0$ is a regular value for $P$, then $\nabla P(x) \not= 0$ whenever $P(x) + c_0 = 0$.

The general case is similar.  We know that $Z  = Z(P_1, ..., P_{n-m})$ is a smooth $m$-dimensional manifold, and $P: Z \rightarrow \RR$ is a smooth function.  By Sard's theorem, almost every $y \in \RR$ is a regular value of the map $P: Z \rightarrow \RR$.  If $x \in Z$ and $P(x)$ is a regular value, then $dP_x \not= 0$, where $dP : T_x Z \rightarrow T_{P(x)} \RR$.  In terms of $\nabla P(x)$, this means that

$$ \nabla P_1(x) \wedge ... \wedge \nabla P_{n-m}(x) \wedge \nabla P(x) \not= 0. $$

\noindent So if $-c_0$ is a regular value for $P: Z \rightarrow \RR$, then $Z(P_1, ..., P_{n-m}, P+c_0)$ is a transverse complete intersection.  
\end{proof}

\subsection{Polynomial partitioning} Polynomial partitioning is a key tool in our arguments.  Our presentation here is a minor variation on the polynomial partitioning result from \cite{GK}.  We begin by stating a partitioning result from \cite{Gu4}:

\begin{theorem} \label{polypartold} (Theorem 1.4 in \cite{Gu4}) Suppose that $W \ge 0$ is a (non-zero) $L^1$ function on $\RR^n$.  Then for each $D$ there a non-zero polynomial $P$ of degree at most $D$ so that $\RR^n \setminus Z(P)$ is a union of $\sim D^n$ disjoint open sets $O_i$, and the integrals $\int_{O_i} W$ are all equal.
\end{theorem}

We want to use this result, but we need to upgrade it in a minor way.  Because we want all the varieties that appear in our argument to be transverse complete intersections, we need to be able to perturb $P$ a little bit.  In order to understand this issue, we need to review some of the proof of Theorem \ref{polypartold}.  The proof is based on the polynomial ham sandwich theorem, which is due to Stone and Tukey \cite{ST}.  Here is a version of the theorem which is convenient for our purposes:

\begin{theorem} \label{polyham}  (Polynomial ham sandwich theorem, cf. Corollary 1.2 in \cite{Gu4}) If $W_1, ..., W_N$ are $L^1$-functions on $\RR^n$, then there exists a non-zero polynomial $P$ of degree $\le C_n N^{1/n}$ so that for each $W_j$,

$$ \int_{\{ P > 0 \} } W_j = \int_{ \{ P < 0 \} } W_j. $$

\end{theorem}

Using the polynomial ham sandwich theorem iteratively, we get the following partitioning result.

\begin{cor} If $W \ge 0$ is a (non-zero) $L^1$-function on $\RR^n$, then there is a sequence of polynomials $Q_1, Q_2, ... $ with $\Deg Q_j \lesssim 2^{j/n}$ with the following equidistribution property:

If $S \ge 1$, and if $\sigma_1, ..., \sigma_S \in \{ -1, +1 \}$ are any sign conditions, then

$$ \int_{ \Sign(Q_s) = \sigma_s \textrm{ for } 1 \le s \le S} W = 2^{-S} \int_{\RR^n} W. $$

\end{cor}

We can slightly perturb each $Q_s$ by adding a small generic constant: $\tilde Q_s = Q_s + c_s$, where $c_s \in \RR$.  Using this small perturbation, we will be able to arrange that all the varieties that appear in our arguments are transverse complete intersections.  As long as the constants $c_s$ are sufficiently small, we still have the following slightly weaker version of the equidistribution result: if $S \ge 1$, and if $\sigma_1, ..., \sigma_S \in \{ -1, +1 \}$ are any sign conditions, then

$$ 2^{-S-1} \int_{\RR^n} W \le \int_{ \Sign(\tilde Q_s) = \sigma_s \textrm{ for } 1 \le s \le S} W \le 2^{-S+1} \int_{\RR^n} W. $$

This gives the following polynomial partitioning result, which is designed to allow small perturbations:

\begin{theorem} \label{polypart} 
Suppose that $W \ge 0$ is a (non-zero) $L^1$ function on $\RR^n$.  Then for any degree $D$ the following holds.

There is a sequence of polynomials $Q_1, ..., Q_S$ with the following properties.  We have $\sum \Deg Q_s \lesssim D$ and $2^S \sim D^n$.  Let $P = \prod_{s=1}^S \tilde Q_s = \prod_{s=1}^S (Q_s + c_s)$ where $c_s \in \RR$.  Let $O_i$ be the open sets given by the sign conditions of $\tilde Q_s$.  There are $2^S \sim D^n$ cells $O_i$ and $\RR^n \setminus Z(P) = \cup_i O_i$.  

If the constants $c_s$ are sufficiently small, then for every $O_i$, 

$$ \int_{O_i} W \sim D^{-n} \int_{\RR^n} W. $$

\end{theorem}

For a generic choice of the constants $c_s$, Lemma \ref{sardtci} guarantees that $Z(\tilde Q_s)$ is a transverse complete intersection for each $s$.  This implies that $Z(P)$ is a finite union of transverse complete intersections.  Similarly, if $Z(P_1, ..., P_{n-m})$ is a transverse complete intersection, then for a generic choice of the constants $c_s$, $Z(P_1, ..., P_{n-m}, \tilde Q_s)$ will also be a transverse complete intersection for each $s$.  

\subsection{Controlling the tangent plane of a variety}

Suppose that $Z$ is an $m$-dimensional transverse complete intersection.  We know that $Z$ is a smooth $m$-dimensional manifold.  We will consider some subsets of $Z$ where the tangent plane obeys certain conditions.  We will see that these subsets are in fact subvarieties of $Z$, and that in generic cases, they are transverse complete intersections.

Let $Z = Z(P_1, ..., P_{n-m})$ be a transverse complete intersection.  Let $w \in \Lambda^m \RR^n$.  Define $Z_w$ by

\begin{equation} \label{defZw} Z_w := \{ x \in Z | \nabla P_1(x) \wedge ... \wedge \nabla P_{n-m}(x) \wedge w = 0 \}. \end{equation}

We note that since $w$ is an $m$-vector, $\nabla P_1(x) \wedge ... \wedge \nabla P_{n-m}(x) \wedge w \in \Lambda^n \RR^n$, which we identify with $\RR$.  Let $g_w := \nabla P_1(x) \wedge ... \wedge \nabla P_{n-m}(x) \wedge w$, a polynomial with degree at most $\Deg P_1 + ... + \Deg P_{n-m}$.  The set $Z_w$ is the algebraic variety $Z(P_1, ..., P_{n-m}, g_w)$.  

\begin{lemma} \label{sardZw} For almost every $w \in \Lambda^m \RR^n$, $Z_w = Z(P_1, ..., P_{n-m}, g_w)$ is a smooth complete intersection.
\end{lemma}

The proof uses some ideas from differential topology.  The book \cite{GP} is a good reference.  In particular, the proof here is closely based on the proof of the transversality theorem in Chapter 2.3 of \cite{GP}.  

\begin{proof} Define a smooth function $g: Z \times \Lambda^m \RR^n \rightarrow \RR$ by

$$ g(x, w) := \nabla P_1(x) \wedge ... \wedge \nabla P_{n-m}(x) \wedge w. $$

The function $g$ is smooth, and it has no critical points, because for any $x \in Z$,  $\nabla P_1(x) \wedge ... \wedge \nabla P_{n-m}(x) \not= 0$, and the restriction of $g$ to $\{ x \} \times \Lambda^m \RR^n$ is a non-zero linear function with no critical points.  Therefore $g^{-1}(0)$ is a smooth submanifold $M$ in $Z \times \Lambda^m \RR^n$ (of codimension 1).

Consider the smooth map $\pi: M \rightarrow \Lambda^m \RR^n$ given by $\pi(x,w) = w$.  Note that $\pi^{-1}(w) = Z_w \times \{w \}$.  We will use $\pi$ in order to study $Z_w$.  We claim that $Z_w$ is a transverse complete intersection whenever $w$ is a regular value of $\pi$.  By Sard's theorem, almost every $w \in \Lambda^m \RR^n$ is a regular value of $\pi$, and so this claim implies our conclusion.

To see that $Z_w$ is a transverse complete intersection, it suffices to check that the critical points of the function $g_w: Z \rightarrow \RR$ are disjoint from $Z_w$.  In other words, it suffices to check that for every $x$ with $(x, w) \in M$, $x$ is a regular point of $g_w: Z \rightarrow \RR$.  If $w$ is a regular value of $\pi$, then it means that for each $x$ with $(x,w) \in M$, $(x,w)$ is a regular point for $\pi$.  So it suffices to check that whenever $(x,w)$ is a regular point for $\pi$, $x$ is a regular point for $g_w$.

Recall that $(x,w) \in M$ is a regular point for $\pi: M \rightarrow \Lambda^m \RR^n$ if and only if $d \pi: T_{(x,w)} M \rightarrow \Lambda^m \RR^n$ is surjective.  To understand this condition better, we compute the tangent space $T_{(x,w)}M$.  We know that $T_{(x,w)} M \subset T_{(x,w)} (Z \times \Lambda^m \RR^n) = T_x Z \times \Lambda^m \RR^n$, and more precisely

$$ T_{(x,w)} M = \{ (v, w') \in T_x Z \times \Lambda^m \RR^n | d g_{(x,w)}(v, w') = 0 \}. $$

But $ dg_{(x,w)} (v, w') = (dg_w)_x (v) + \nabla P_1(x) \wedge ... \wedge \nabla P_{n-m}(x) \wedge w'$.  Therefore, 

$$ T_{(x,w)} M = \{ (v, w') \in T_x Z \times \Lambda^m \RR^n | (d g_w)_x (v) + \nabla P_1(x) \wedge ... \wedge \nabla P_{n-m}(x) \wedge w' = 0 \}. $$

If $x$ is not a regular point of $d g_w$, then $(d g_w)_x = 0$, and $T_{(x,w)} M = \{ (v, w') | \nabla P_1(x) \wedge ... \wedge \nabla P_{n-m}(x) \wedge w' = 0 \}$.  But in this case, the projection $d \pi: T_{(x,w)} M \rightarrow \Lambda^m \RR^n$ is not surjective.  (The projection $d \pi$ is just $d \pi(v, w') = w'$.)  So if $(x,w)$ is a regular point of $\pi$, then $x$ is a regular point of $g_w$ as desired.
\end{proof}

If $W \subset \Lambda^m \RR^n$ is a large finite set, then on each connected component of $Z \setminus (\cup_{w \in W} Z_w)$, the tangent plane $TZ$ is constrained in a small region of the Grassmannian.  More precisely, for any small parameter $\beta > 0$, we can choose a finite set $W \subset \Lambda^m \RR^n$ so that, for any two points $x_1, x_2$ in the same component of $Z \setminus (\cup_{w \in W} Z_w)$, $\Angle (T_{x_1} Z, T_{x_2} Z) < \beta$.  We can also choose $W$ generically so that each $Z_w$ is a transverse complete intersection of dimension $m-1$.

\subsection{Controlling transverse intersections between a tube and a variety} Suppose that $T$ is a cylinder of radius $r$ with central axis $\ell$.  Suppose that $Z^m \subset \RR^n$ is a transverse complete intersection.  Define $Z_{> \alpha}$ by 

$$ Z_{> \alpha} := \{ z \in Z | \Angle ( T_z Z, \ell) > \alpha \}. $$

\begin{lemma} \label{transinterbound} Suppose that $Z = Z(P_1, ..., P_{n-m})$ is a transverse complete intersection and that the polynomials $P_j$ have degree at most $D$.  Let $T$ be a tube of radius $r$ as above.  Then for any $\alpha > 0$, $Z_{> \alpha} \cap T$ is contained in a union of $\lesssim D^n$ balls of radius $\lesssim r \alpha^{-1}$.
\end{lemma}

The main tool in the proof is the following version of the Bezout theorem:

\begin{theorem} \label{bezout} (cf. Theorem 5.2 of \cite{CKW} for a short proof) Suppose that $Z = Z(Q_1, ..., Q_n)$ is a transverse complete intersection in $\RR^n$.  Then $Z$ is finite and the cardinality of $Z$ is at most $\prod_{j=1}^n \Deg Q_j$. 
\end{theorem}

Using this Bezout theorem, we now prove Lemma \ref{transinterbound}.

\begin{proof} The proof is by induction on $m$.  When $m=0$, Theorem \ref{bezout} guarantees that $Z$ consists of at most $D^n$ points, and the conclusion follows.

Now we turn to the inductive step.  Without loss of generality, we can assume that $\ell$ is the $x_n$-axis.  We let $T_r$ denote the $r$-neighborhood of the $x_n$-axis.  

Next we do some scaling to reduce to a special case.  By rescaling, we can reduce to the case that $r=1$.  Next, by scaling in the $x_n$-coordinate only, we can reduce to the case that $\alpha = 1$.  
So we have to show that $Z_{> 1} \cap T_1$ is contained in $\lesssim D^n$ balls of radius $\lesssim 1$.





Let $Z_w$ be defined as in the last subsection.  We choose $\lesssim 1$ values of $w$ in general position so that on each connected component of $Z \setminus \cup_w Z_w$, the tangent plane of $Z$ varies by an angle at most $\frac{1}{100}$.  Since $w$ is generic, $Z_w = Z(P_1, ..., P_{n-m}, g_w)$ is a transverse complete intersection of dimension $m-1$.  Also $\Deg g_w \lesssim D$.  We can apply our inductive assumption to $Z_w$, using radius $r=20$ and $\alpha = 1/2$.  We see that
$ Z_{w, > 1/2} \cap T_{20}$ is covered by $\lesssim D^n$ balls of radius $\lesssim 1$.

Next we claim that $Z_{>1/2} \cap Z_{w} \subset Z_{w, > 1/2}$.  If $x \in Z_{>1/2} \cap Z_w$, then

$$ \inf_{v \in T_x Z} \Angle(v, \ell) > 1/2. $$

\noindent But $T_x Z_w \subset T_x Z$.  Therefore, $\inf_{v \in T_x Z_w} \Angle(v, \ell) > 1/2$, and so $x \in Z_{w, >1/2}$ as claimed.   

Since the total number of $w$ is $\lesssim 1$, 

$$\cup_w Z_{> 1/2} \cap Z_w \cap T_{20} \textrm{ is contained in $\lesssim D^n$ balls $B_i$ of radius $\lesssim 1$}. $$

If $B_i = B(x_i, r_i)$ we write $10 B_i$ for $B(x_i, 10 r_i)$.  The union $\cup_i 10 B_i$ is a set of $\lesssim D^n$ balls of radius $\lesssim 1$ which covers part of $Z_{> 1} \cap T_1$.  We still have to cover the remaining part

$$  Z_{>1} \cap T_1 \setminus \cup_i 10 B_i. $$

Consider a point $z$ in this remaining part.  We can assume the radius of each $B_i$ is at least 2, and so we know that the distance from $z$ to $\cup_i B_i$ is at least 10.  Let $A$ be the connected component of $Z \cap B(z,10)$ containing $z$.  We claim that $A$ is disjoint from all $Z_w$.  Indeed, suppose that $\gamma$ was a curve in $A$ starting at $z$ and intersecting $\cup_w Z_w$ for the first time at $z' \in A$.  Along the curve $\gamma$, the tangent plane of $TZ$ is constant up to angle $\frac{1}{100}$.  Since $\gamma$ starts at $z \in Z_{>1}$,  $\gamma \subset Z_{> 1/2}$.  Also, $\gamma \subset B(z,10) \subset T_{20}$.  We conclude that $z' \in Z_{>1/2} \cap Z_w \cap T_{20}$, and so so $z' \in \cup_i B_i$.  But $B(z,10)$ is disjoint from $\cup_i B_i$.  This contradiction proves the claim.  Since $A$ is connected and disjoint from all $Z_w$, the tangent plane of $Z$ is constant on $A$ up to angle $\frac{1}{100}$.  

Therefore, $A$ is a small perturbation of an $m$-plane that cuts across $T_1$ in a quantitatively transverse way.  Let $\Pi$ be a random $(n-m)$-plane containing the $x_n$-axis.  With probability $\gtrsim 1$, $\Pi \cap A \cap T_{1}$ is non-empty.  By the Bezout theorem (Theorem \ref{bezout}), $|\Pi \cap Z| \le D^{n-m}$ for generic $\Pi$.  Therefore, there can be at most $\lesssim D^{n-m}$ disjoint sets $A$ of this type.  So we see that the remaining part of $Z_{>1} \cap T_1$ is contained in $\lesssim D^{n-m}$ additional balls of radius $\lesssim 1$.  

\end{proof}

\section{Transverse equidistribution estimates}

In this section we prove a transverse equidistribution estimate.  To set up the statement, we first define what it means for a wave packet to be tangent to a transverse complete intersection $Z$.

\begin{defn} \label{defTZ}

Suppose that $Z = Z(P_1, ..., P_{n-m})$ is a transverse complete intersection.  We say that $\Ttv$ is $R^{-1/2 + \dm}$-tangent to $Z$ in $B_R$ if:

\begin{equation} \label{tangZdist} \Ttv \subset \NZR, \end{equation}

\noindent and for any $x \in \Ttv$ and $z \in Z \cap B_R$ with $|x-z| \lesssim R^{1/2 + \dm}$, 

\begin{equation} \label{tangZangle} \Angle (G(\theta), T_z Z ) \lesssim R^{-1/2 + \dm}. \end{equation}

We define

$$ \TTT_Z := \{ (\tv) | \Ttv \textrm{ is $R^{-1/2 + \dm}$ tangent to } Z \textrm{ in } B_R \}. $$

We say that $f$ is concentrated in wave packets from $\TTT_Z$ if

$$ \sum_{(\tv) \notin \TTT_Z} \| \ftv \|_{L^2} \le \RD(R) \| f \|_{L^2}. $$

\end{defn}

(In this definition $\dm > 0$ is a small constant.  The estimates we prove in this section hold for any $\dm \ge 0$.  We will choose $\dm$ in Section \ref{secmainind}.) 

Suppose that $B$ is a ball of radius $R^{1/2 + \dm}$ in $\RR^n$.  Define

\begin{equation} \label{defTBZ}  \TTT_{B,Z} := \{ (\tv) \in \TTT_Z | \Ttv \cap B \not= \emptyset \}. \end{equation}

The main result of this section is the following transverse equidistribution estimate.  

\begin{lemma} \label{equid2} Suppose that $B$ is a ball of radius $R^{1/2 + \dm}$ in $B_R \subset \RR^n$ and 
that $\rho \le R$.  Suppose that  $g = \sum_{(\tv) \in \TTT_{B,Z}} \gtv$.  Then

$$ \int_{B \cap N_{\rho^{1/2 + \dm}}(Z)} | \EP g |^2 \lesssim R^{O(\dm)} \left(\frac{R^{1/2}}{\rho^{1/2}}\right)^{-(n-m)}\int_{2B} | \EP g |^2 + \RD(R) \| g \|_{L^2}^2. $$

\end{lemma}

We build up to the proof via several smaller lemmas.  We begin with a version of the Heisenberg uncertainty principle, saying that a function which is concentrated in a small ball in frequency space cannot concentrate too much in physical space.

\begin{lemma} \label{heis1} Suppose that $G: \RR^{n} \rightarrow \CC$ is a function, and that $\hat G$ is supported in a ball of radius $r$, $B(\xi_0, r)$.  Then for any ball $B(x_0, \rho)$ of radius $\rho \le r^{-1}$,

$$ \int_{B(x_0, \rho)} |G|^2 \lesssim \frac{|B_\rho|}{|B_{r^{-1}}|} \int |G|^2. $$
\end{lemma}

\begin{proof} Let $\eta$ be a smooth bump function with $|\eta| \sim 1$ on $B(x_0, \rho)$ and rapidly decaying outside of it.  Then $| \hat \eta(\xi) | \sim |B_\rho|$ on $B_{\rho^{-1}}$ and rapidly decaying outside of it.

$$ \int_{B_\rho} |G|^2 \lesssim \int | \eta G|^2 = \int | \hat \eta * \hat G|^2. $$

For $\xi \lesssim \rho^{-1}$, we bound

$$ | \hat \eta * \hat G (\xi) | \le \| \hat \eta \|_{L^\infty} \| \hat G \|_{L^1} \sim | B_\rho | \int_{B_r} |\hat G|. $$

For $|\xi|$ far from $B_{\rho^{-1}}$, the rapid decay of $\eta$ takes over and gives a stronger bound.  All together, we have

$$  \int | \hat \eta * \hat G|^2 \lesssim |B_\rho^{-1}| \left( |B_\rho| \int_{B_r} |\hat G| \right)^2 \le |B_\rho| |B_r| \int | \hat G|^2 = \frac{| B_\rho|}{|B_{r^{-1}}|} \int |G|^2. $$

\end{proof}

Next we need a more local version of this lemma.

\begin{lemma} \label{heis2} Suppose that $G: \RR^{n} \rightarrow \CC$ is a function, and that $\hat G$ is supported in a ball of radius $r$, $B(\xi_0, r)$.  Then for any ball $B(x_0, \rho)$ with $\rho \le r^{-1}$, we have the inequality
 
$$ \int_{B(x_0, \rho)} |G|^2 \lesssim \frac{|B_\rho|}{|B_{r^{-1}}|} \int W_{B(x_0, r^{-1})} |G|^2 ,  $$

where $W_{B(x_0, r^{-1})}$ is a weight function which is equal to 1 on $B(x_0, r^{-1})$ and rapidly decaying outside of it.  
\end{lemma} 

\begin{proof} Let $\psi$ be a function with the support of $\hat \psi \subset B_r$ and with $|\psi| \sim 1$ on $B(x_0, r^{-1})$ and rapidly decaying outside $B(x_0, r^{-1})$.  Our weight function will be $W = |\psi|^2$.

Let $H = \psi \cdot G$.  Note that $\hat H$ is supported in $B(\xi_0, 2r)$.  Applying Lemma \ref{heis1} to $H$, we see that

$$ \int_{B(x_0, \rho)} |G|^2 \lesssim \int_{B(x_0, \rho)} |H|^2 \lesssim \frac{|B_\rho|}{|B_{r^{-1}}|} \int  |H|^2 = \frac{|B_\rho|}{|B_{r^{-1}}|} \int  W |G|^2. $$

\end{proof}

Suppose that $B$ is a ball of radius $R^{1/2 + \dm}$ in $\RR^n$, and $V$ is a subspace of $\RR^n$.  Define

\begin{equation} \label{defTBV}  \TTT_{B,V} := \{ (\tv) | \Ttv \cap B \not= \emptyset \textrm{ and } \Angle( G(\theta), V ) \lesssim R^{-1/2 + \dm} \}. \end{equation}

\noindent Let $2B$ denote the ball with the same center as $B$ and twice the radius.

If $g = \sum_{(\tv) \in \TTT_{B,V}} \gtv$, then we will show that $\EP g$ is equidistributed in $B$ along directions transverse to $V$.  More precisely, we have the following lemma.  

\begin{lemma} \label{equid1} If $V \subset \RR^n$ is a subspace, then there is a subspace $V' \subset \RR^n$
with the following properties.

\begin{enumerate}

\item $\Dim V + \Dim V' = n$.  

\item $V'$ is transverse to $V$ in the sense that for any unit vectors $v \in V$ and $v' \in V'$,

$$ \Angle(v, v') \gtrsim 1. $$

\item Suppose that 

$$g = \sum_{(\tv) \in \TTT_{B,V}} \gtv. $$

If $\Pi$ is a plane parallel to $V'$, and $x_0 \in \Pi \cap B$, then for any $\rho \le R$, 

\begin{equation} \label{equibound1} \int_{\Pi \cap B(x_0, \rho^{1/2 + \dm})} | \EP g |^2 \lesssim R^{O(\dm)}  \left( \frac{R^{1/2}}{\rho^{1/2}} \right)^{-\Dim V'} \int_{\Pi \cap 2B}  | \EP g |^2 + \RD(R) \| g \|_{L^2}^2. \end{equation}

\end{enumerate}

\end{lemma}

\begin{proof} To prove the lemma, we locate the appropriate space $V'$ and then we appeal to Lemma \ref{heis2}.  

The proof is about relating the direction $G(\theta)$ and the frequency $\xi(\theta)$.  Whenever an object has to do with directions, we label it with a subscript $G$, and whenever it has to do with frequencies, we label it with a subscript $\xi$.   Since $V$ is a set of directions of tubes, we write $V_G$ for $V$.   Consider the set of $\omega \in B^{n-1}$ so that $\Angle ( G(\omega), V_G) \lesssim \alpha$, for some small $\alpha$.  What can we say about the frequencies $\xi (\omega)$?

Recall that $G(\omega) =  \frac{ G_0(\omega)}{|G_0(\omega)|}$, where

$$ G_0(\omega) = ( - 2 \omega_1, ..., -2 \omega_{n-1}, 1). $$

Let $\RR^{n-1} \subset \RR^n$ be the $(x_1, ..., x_{n-1})$-plane.  Examining the formula for $G_0$ and $G$, we see that for every $\omega \in B^{n-1}$, $\Angle( G(\omega), \RR^{n-1} ) \ge c_{\textrm{angle}} > 0$.

We define

$$ \Angle( V_G, \RR^{n-1} ) := \max_{v \in V_G} \Angle (v, \RR^{n-1} ). $$

If $\Angle (V_G, \RR^{n-1})$ is smaller than $(1/2) c_{\textrm{angle}}$, then $G(\omega)$ is never close to $V_G$ for any $\omega \in B^{n-1}$.  In this case $\TTT_{B,V}$ is empty and there is nothing to prove.  Therefore, we can assume from now on that

\begin{equation} \label{angvG1} \Angle (V_G, \RR^{n-1}) \gtrsim 1. \end{equation}

We note that $G(\omega) \in V_G$ if and only if $G_0(\omega) \in V_G$.  We define $A_{\omega} \subset \RR^{n-1}$ by

$$ A_{\omega} := \{ \omega \in \RR^{n-1} | G_0(\omega) \in V_G \}. $$

Since $G_0$ is an affine map, $A_\omega$ is an affine subspace of $\RR^{n-1}$.  Since $V_G$ is quantitatively transverse to $\RR^{n-1}$, and since $G_0(\omega) = (-2 \omega_1, ..., -2 \omega_{n-1}, 1)$, we see that $\Dim A_\omega = \dim V -1$.  We also see that for $\omega \in B^{n-1}$, 

$$ \Dist ( \omega, A_\omega) \lesssim \Angle( G(\omega), V_G ). $$

Next we define $A_\xi$ to be $A_\omega \times \RR \subset \RR^n$, an affine subspace of $\RR^n$.  Since $\xi(\omega) = (\omega_1, ..., \omega_{n-1}, |\omega|^2)$, we see that

$$ \Dist( \xi(\omega), A_\xi) = \Dist (\omega, A_\omega). $$

The spaces $A_\omega$ and $A_\xi$ are affine spaces.  We let $V_\omega$ and $V_\xi$ be the subspaces parallel to $A_\omega$ and $A_\xi$.  The space $V_{\xi}^\perp \subset \RR^n$ is our space $V'$.  The dimension of $V_\xi = \Dim A_\omega + 1 = \Dim V$, and so $\Dim V' + \Dim V = n$ as desired.  

We let $\pi_{V'}$ be the orthogonal projection from $\RR^n$ to $V' = V_\xi^\perp$.  Combining our estimates above, we see that

$$ \pi_{V'} \left( \{ \omega \in B^{n-1} | \Angle( G(\omega), V_G ) \le \alpha \} \right) \subset \textrm{ a ball of radius } \lesssim \alpha. $$

Next we want to see that $V_G$ and $V_\xi^\perp$ are quantitatively transverse.  We define

$$ \Angle (V_G , V_\xi^\perp ) := \min_{v \in V_G, w \in v_\xi^\perp} \Angle(v,w). $$

\begin{sublemma} 
$$\Angle(V_G, \RR^{n-1}) = \Angle(V_G, V_\xi^\perp).  $$  
\end{sublemma}

In particular, in the non-vacuous case that $\Angle(V_G, \RR^{n-1}) \gtrsim 1$, we see that $V = V_G$ and $V' = V_\xi^\perp$ are quantitatively transverse.

\begin{proof} The intersection $V_G \cap G_0(\RR^{n-1})$ is an affine space parallel to $V_\omega$.  Let $v_G \in V_G$ be a unit vector perpendicular to $V_\omega$.  Let $v_1, ..., v_{m-1}$ be an orthonormal basis of $V_\omega$.  Then $v_1, ..., v_{m-1}, v_G$ is an orthonormal basis of $V_G$.

Let $e_n$ be the $n^{th}$ coordinate unit vector.  We see that $v_1, ..., v_{m-1}, e_n$ is an orthonormal basis for $V_\xi$.  We also see that $V_{\xi}^\perp \subset \RR^{n-1}$ and $V_\xi^\perp = V_\omega^\perp \subset \RR^{n-1}$.  

Since $v_1, ..., v_{m-1} \subset V_\omega \subset \RR^{n-1}$,

$$\Angle(V_G, \RR^{n-1}) = \Angle( v_G, \RR^{n-1}).$$

Since $v_G$ is perpendicular to $V_\omega$, we see that the projection of $v_G$ to $\RR^{n-1}$ actually lies in $V_{\xi}^\perp$, and so 

$$ \Angle(v_G, \RR^{n-1}) = \Angle(v_G, V_\xi^\perp). $$

But since $v_1, ..., v_{m-1}$ are in $V_\xi$, we see that $v_G$ is the vector in $V_G$ which makes the smallest angle with $V_\xi^\perp$, and so

$$ \Angle(v_G, V_\xi^\perp) = \Angle(V_G, V_\xi^\perp). $$

\end{proof}

Let $\Pi$ be an $(n-m)$-plane parallel to $V'$ passing through $B$.  We know that $(\EP \gtv)^\wedge$ is supported in $\xi(\theta)$.  The restriction to $\Pi$ of $\EP \gtv$ has Fourier transform supported in $\pi_{V'} (\xi(\theta))$.  Now for all $(\theta, v) \in \TTT_{B,V}$, $\Angle( G(\theta), V ) \lesssim R^{-1/2 + \dm}$, and so all $\pi_{V'} (\xi(\theta))$ lie in a single ball of radius $\lesssim R^{-1/2 + \dm}$.  Therefore, if we view $\EP g$ as a function $G: \Pi \rightarrow \CC$, its Fourier transform is supported in a ball of radius $\lesssim R^{-1/2 + \dm}$.  We apply Lemma \ref{heis2}, giving:

\begin{equation} \label{equidonPi} \int_{\Pi \cap B(x_0, \rho^{1/2 + \dm})} | \EP g |^2 \lesssim \left( \frac{R^{1/2 - \dm}}{\rho^{1/2 + \dm}} \right)^{-\Dim V'} \int_\Pi W_{B(x_0, R^{1/2 - \dm})}  | \EP g |^2  \end{equation}

$$ \lesssim R^{O(\dm)}  \left( \frac{R^{1/2}}{\rho^{1/2}} \right)^{-\Dim V'} \int_\Pi W_B  | \EP g |^2. $$

Finally, $\EP g = \sum_{(\tv) \in \TTT_{B,V}} \EP \gtv$.  Each $\EP \gtv$ is essentially supported on $\Ttv$.  Since $\Ttv$ is transverse to $\Pi$ and intersects $B$, we see that if $x \in \Pi \setminus 2B$, $| \EP g(x) | \le \RD(R) \| g \|_{L^2}$.  So

$$ \int_\Pi W_B  | \EP g |^2 \le \int_{\Pi \cap 2B} | \EP g |^2 + \RD(R) \| g \|_{L^2}^2. $$

\end{proof} 

Now we are ready to prove Lemma \ref{equid2}.  

\begin{proof} Since $(\tv) \in \TTT_Z$, and $\Ttv \cap B$ is non-empty, we know that for any $z \in Z \cap 2B$,

$$ \Angle(T_z Z , G(\theta)) \lesssim R^{-1/2 + \dm}. $$

Let $V$ be a subspace of lowest possible dimension so that, for all $(\tv) \in \TTT_{B,Z}$, 

$$ \Angle( V, G(\theta)) \lesssim R^{-1/2 + \dm}. $$

Let $V'$ be the subspace given by Lemma \ref{equid1}.  We know that $\Dim V + \Dim V' = n$, and we know that $V'$ is quantitatively transverse to $V$.  By (\ref{equidonPi}), we also know that for any plane $\Pi$ parallel to $V'$, 

\begin{equation} \label{intgrhoinPi} \int_{\Pi \cap B(x_0, \rho^{1/2 + \dm})} | \EP g |^2 \lesssim R^{O(\dm)}  \left( \frac{R^{1/2}}{\rho^{1/2}} \right)^{-\Dim V'} \int_{\Pi \cap 2B}  | \EP g |^2 + \RD(R) \| g \|_{L^2}^2. \end{equation}

We claim that for each $z \in Z \cap B$, $T_z Z$ is quantitatively transverse to $V'$.  If this is not the case, it means that there exists a point $z \in Z$ and a subspace $W \subset T_z Z$ with

$$ \Dim W > \Dim Z - \Dim V, $$

so that for each non-zero $w \in W$,

$$ \Angle(w, V') \le o(1). $$

Since $V$ and $V'$ are transverse, this angle condition guarantees that

$$ \Angle(w, V) \gtrsim 1. $$

Because of this angle condition, we can construct a linear map $L: \RR^n \rightarrow V$ so that $L$ restricted to $V$ is the identity, $L$ restricted to $W$ is zero, and $|L| \lesssim 1$.  Recall that for each $(\tv) \in \TTT_{B, Z}$, $\Angle( G(\theta), V) \lesssim R^{-1/2 + \dm}$, and so

$$ | L ( G(\theta) ) - G(\theta) | \lesssim R^{-1/2 + \dm}. $$

On the other hand, we know that $G(\theta) \subset N_{R^{-1/2 + \dm}}(T_z Z) \cap B(1)$, and so $L(G(\theta))$ lies in

$$ L( N_{R^{-1/2 + \dm}}(T_z Z) \cap B(1) ) \subset N_{C R^{-1/2 + \dm}} ( L ( T_z Z) ). $$

This shows that for all $(\theta, v) \in \TTT_{B, Z}$, 

$$ \Angle ( G(\theta), L (T_z Z) ) \lesssim R^{-1/2 + \dm}. $$

But since $L$ vanishes on $W$, $L ( T_z Z )$ is a subspace of dimension at most $\Dim Z - \Dim W < \Dim V$.  This contradicts our hypothesis that $V$ has minimal dimension.  This finishes the proof of our claim that for each $z \in Z \cap 2B$, $T_z Z$ is quantitatively transverse to $V'$.

Suppose that $\Pi$ is a plane parallel to $V'$ and intersecting $B$.  Given the transversality we just proved, it follows that

$$ \Pi \cap N_{\rho^{1/2 + \dm}}(Z) \cap B \subset N_{C \rho^{1/2 + \dm}}(\Pi \cap Z) \cap \Pi \cap 2B. $$

Note that $\Pi \cap Z$ is itself a transverse complete intersection of dimension $\Dim V' + \Dim Z - n$.  
Now the set, $N_{C \rho^{1/2 + \dm}}(\Pi \cap Z) \cap \Pi \cap 2B$, can be covered by $R^{O(\dm)} 
\left( \frac{R^{1/2}}{\rho^{1/2}} \right)^{\Dim V' + \Dim Z - n}$ balls in $\Pi$ of radius $\rho^{1/2 + \dm}$ (cf. \cite{W}).  Applying (\ref{intgrhoinPi}) on each of these balls and summing, we get the bound

\begin{equation*}
\int_{\Pi \cap N_{\rho^{1/2 + \dm}}(Z) \cap B} | \EP g |^2 \lesssim R^{O(\dm)}  \left( \frac{R^{1/2}}{\rho^{1/2}} \right)^{-(n-m)} \int_{\Pi \cap 2B}  | \EP g |^2 + \RD(R) \| g \|_{L^2}^2. 
\end{equation*}

Finally, integrating over planes $\Pi$ parallel to $V'$ (using Fubini) we get the desired bound:

\begin{equation*}
\int_{B \cap N_{\rho^{1/2 + \dm}}(Z)} | \EP g |^2 \lesssim R^{O(\dm)}  \left( \frac{R^{1/2}}{\rho^{1/2}} \right)^{-(n-m)} \int_{2B}  | \EP g |^2 + \RD(R) \| g \|_{L^2}^2. 
\end{equation*}

\end{proof}

\section{Adjusting a wave packet decomposition to a smaller ball} \label{secsmallerball}

Suppose that $B(y, \rho) \subset B_R$ for some radius $\rho$ in the range $R^{1/2} < \rho < R$, and we want to decompose $f$ into wave packets associated to the ball $B(y, \rho)$.  How does the new wave packet decomposition relate to the old wave packet decomposition?  

If the center $y$ is not at the origin, then we introduce new coordinates

$$\tx = x - y. $$

We define

$$ \psi_y(\omega) := e^{i ( y_1 \omega_1 + ... + y_{n-1} \omega_{n-1} + y_n |\omega|^2)}. $$

Then we write

$$ \EP f(x) = \int e^{i ( x_1 \omega_1 + ... + x_{n-1} \omega_{n-1} + x_n |\omega|^2)} f(\omega) = 
\int e^{i ( \tx_1 \omega_1 + ... + \tx_{n-1} \omega_{n-1} + \tx_n |\omega|^2)} e^{i \psi_y(\omega)} f(\omega). $$

For any function $f$, we use the notation 

$$ \tf(\omega) = e^{i \psi_y(\omega)} f(\omega). $$

In this notation, we now have

$$ \EP f(x) = \EP \tf (\tx). $$

Next we decompose $\tf$ into wave packets adapted to the ball $B_\rho$.  We follow the construction of wave packets in Section \ref{basicwavepack}, except with the radius $R$ replaced by $\rho$.  We cover $B^{n-1}$ with caps $\tith$ of radius $\rho^{-1/2}$.  We cover $\RR^{n-1}$ by finitely overlapping balls of radius $\sim \rho^{\frac{1 + \dt}{2}}$, centered at vectors $\tiv \in \rho^{\frac{1 + \dt}{2}} \ZZ^{n-1}$.  And we decompose $\tf$ as 

$$ \tf = \sum_{\ttv} \tftv + \RD(R) \| f \|_{L^2}, $$

\noindent where $\tftv$ is supported in $\tith$ and its Fourier transform is essentially supported in $B(\tiv, \rho^{\frac{1 + \dt}{2}})$. 
For each $(\ttv)$, $\EP \tftv$ is essentially supported on a tube $\tTtv$ of radius $\rho^{1/2 + \dt}$ and length $\rho$.  In the $\tx$ coordinates, this tube is contained in $B_\rho$, while in the original $x$ coordinates, this tube is contained in $B(y, \rho)$. 

How does the original wave packet decomposition $f = \sum_{\tv} \ftv$ relate to the new one?  The first question we study is, if we expand $\ftv$ in wave packets at scale $B_\rho$, $(\ftv)^\sim = \sum_{\ttv} (\ftv)^\sim_{\ttv}$, then which $(\ttv)$ can have a significant contribution?  We answer this question in Lemma \ref{scalesftvtith}.  Before stating the lemma, we need a couple definitions.

For a given $y$ and $\omega$ we define 

$$\bar v(\omega, y) := \partial_\omega \psi_y(\omega),$$

and we compute

$$ \bar v(\omega, y) = \partial_\omega \psi_y(\omega) = \partial_\omega ( y_1 \omega_1 + ... + y_{n-1} \omega_{n-1} + y_n |\omega|^2) $$

$$ = (y_1 + 2 \omega_1 y_n, ..., y_{n-1} + 2 \omega_{n-1} y_n) = y' + 2 y_n \omega. $$

\noindent (Here we use the notation $y' = (y_1, ..., y_{n-1})$.)  If $\omega_\theta$ denotes the center of a cap $\theta$, we also write $\bar v(\theta, y)$ for $\bar v(\omega_\theta, y)$.  

Define

\begin{equation} \label{deftiTTTthv}
\tilde \TTT_{\tv} := \{ (\ttv) : \Dist(\theta, \tith) \lesssim \rho^{-1/2} \textrm{ and } |v + \bar v(\theta, y) - \tiv| \lesssim R^{1/2 + \delta/2} \}.
\end{equation}

\begin{lemma} \label{scalesftvtith} The function $(\ftv)^\sim$ is concentrated in wave packets from $\tilde \TTT_{\tv}$.  In other words,

$$ (\ftv)^{\sim} = \sum_{(\ttv) \in \tilde \TTT_{\tv}} (\ftv)^{\sim}_{\ttv} + \RD(R) \| f \|_{L^2}. $$

\end{lemma}

\begin{proof} Since $\ftv$ is supported in $\theta$, the support of $(\ftv)^{\sim}$ is clearly contained in 

$$ \cup \{ \tith : \Dist(\tith, \theta) \lesssim \rho^{-1/2} \}. $$

The main point is to check that the Fourier transform of $(\ftv)^{\sim}$ is essentially supported in a ball around $v + \bar v(\theta,y)$ of radius $\lesssim R^{1/2 + \dt/2}$.  Let $\eta_\theta$ be a bump function which is 1 on $\theta$ and decays to 0 outside of $2 \theta$.  Then

$$ \left( e^{i \psi_y(\omega)} \ftv \right)^\wedge =  \left( \eta_\theta e^{i \psi_y(\omega)} \cdot \ftv \right)^\wedge = 
\left( \eta_\theta e^{i \psi_y(\omega)} \right)^\wedge * \left( \ftv \right)^\wedge . $$

Now $(\ftv)^\wedge$ is rapidly decaying outside of $B(v, R^{1/2 + \dt/2})$.  On the other hand, a stationary phase argument shows that $\left( \eta_\theta e^{i \psi_y(\omega)} \right)^\wedge$ is rapidly decaying outside of $B( \bar v(\theta, y), R^{1/2})$.  (To see this, it helps to note that on the support of $\eta_\theta$, $\partial_\omega \psi_y$ lies in a ball around $\bar v(\theta, y)$ of radius $\lesssim R^{1/2}$.)
\end{proof}

Next we explore the geometric features of a tube $\tTtv$ with $(\ttv) \in \tilde \TTT_{\tv}$.

\begin{lemma} \label{tubetildetube} If $(\ttv) \in \tilde \TTT_{\tv}$, then the tube $\tTtv$ obeys the following geometric estimates:
 
\begin{equation} \label{tubetildetubedist} \HausDist(\tTtv, \Ttv \cap B(y, \rho) ) \lesssim R^{1/2 + \dt}. \end{equation}

\noindent and 

\begin{equation} \label{tubetildetubeangle} \Angle (G(\theta), G(\tith) ) \lesssim \rho^{-1/2}. \end{equation}

\end{lemma}

\begin{proof} We recall the definition of $\Ttv$ from (\ref{defTtv}):

$$\Ttv := \{ (x', x_n) \in B_R \textrm{ so that } |x' + 2 x_n \omega_\theta + v | \le R^{1/2 + \dt} \}. $$

In the coordinates $\tx$, since $x = \tx + y$ and $y' + 2 y_n \omega_\theta = \bar v(\theta, y)$, 

\begin{equation} \label{TtvcapBy} \Ttv \cap B(y, \rho) = \{ (\tx', \tx_n) \in B_\rho \textrm{ so that } |\tx' + 2 \tx_n \omega_\theta + \bar v(\theta, y) + v | \le R^{1/2 + \dt} \}. \end{equation}

On the other hand, 

\begin{equation} \label{tTtvcapBy} \tTtv :=  \{ (\tx', \tx_n) \in B_\rho \textrm{ so that } |\tx' + 2 \tx_n \omega_{\tith} + \tiv | \le \rho^{1/2 + \dt} \}. \end{equation}

By the definiton of $\tilde \TTT_{\tv}$, $\Dist(\tith, \theta) \le \rho^{-1/2}$ and so $| \omega_\theta - \omega_{\tith}| \lesssim \rho^{-1/2}$.  Since $|\tx_n| \le \rho$, $| 2 \tx_n \omega_\theta - 2 \tx_n \omega_{\tith}| \lesssim \rho^{1/2}$.
By the definition of $\tilde \TTT_{\tv}$, $| v + \bar v(\theta, y) - \tiv | \lesssim R^{1/2 + \dt/2}$.  Comparing (\ref{TtvcapBy}) and (\ref{tTtvcapBy}), we see that $\HausDist(\Ttv \cap B(y, \rho), \tTtv) \lesssim R^{1/2 + \dt}$ as desired.

Since $\Dist(\tith, \theta) \le \rho^{-1/2}$ it follows that $\Angle ( G(\theta), G(\tith) ) \lesssim \rho^{-1/2}$.

\end{proof}

Many different $(\theta, v)$ lead to essentially the same set $\tilde \TTT_{\tv}$.  If $\Dist(\theta_1, \theta_2) \le \rho^{-1/2}$, and $| v_1 + \bar v(\theta_1, y) - v_2 - \bar v(\theta_2, y) | \le R^{1/2 + \dt/2}$, then $\tilde \TTT_{\theta_1, v_1}$ and $\tilde \TTT_{\theta_2, v_2}$ are essentially the same.  We can organize the possible pairs $(\tv)$ into equivalence classes in the following way.  If $\tith$ is one of our caps of radius $\rho^{-1/2}$, and $w \in R^{1/2 + \delta/2} \ZZ^{n-1}$, then we define

\begin{equation} \label{defTTTtithtw} \TTT_{\tith, w} := \{ (\theta, v) : \Dist(\theta, \tith) \lesssim \rho^{-1/2} \textrm{ and } |v + \bar v(\theta, y) - w | \lesssim R^{1/2 + \delta/2} \} . \end{equation}

\noindent If $(\theta_1, v_1)$ and $(\theta_2, v_2)$ lie in the same set $\TTT_{\tith, w}$, then $\tilde \TTT_{\theta_1, v_1}$ and $\tilde \TTT_{\theta_2, v_2}$ are essentially the same.  They are both contained in (and essentially equal to)

\begin{equation} \label{deftiTTTtithtw} \tilde \TTT_{\tith, w} := \{ (\tilde \theta_1, \tilde v) : \Dist(\tilde \theta_1, \tith) \lesssim \rho^{-1/2} \textrm{ and } |w - \tilde v | \lesssim R^{1/2 + \delta/2} \} . \end{equation}

Now Lemma \ref{scalesftvtith} gives the following corollary.

\begin{lemma} \label{scalestithv} If $g$ is concentrated in wave packets in $\TTT_{\tith, w}$, then $\tilde g$ is concentrated in wave packets in $\tilde \TTT_{\tith, w}$.  In other words, if

$$ g = \sum_{(\tv) \in \TTT_{\tith, w}} \gtv + \RD(R) \| g \|_{L^2}, $$

then

$$ \tilde g = \sum_{(\ttv) \in \tilde \TTT_{\tith, w}} \tgtv + \RD(R) \| g \|_{L^2}. $$

\end{lemma}

We also note that the sets $\TTT_{\tith, w}$ are essentially disjoint, and their union contains all the possible pairs $(\tv)$.  Similarly, the sets $\tilde \TTT_{\tith, w}$ are essentially disjoint and their union contains all possible pairs $(\ttv)$.  With this in mind, we define

\begin{equation} \label{defgtithv} 
g_{\tith, w} := \sum_{(\tv) \in \TTT_{\tith, w}} \gtv. 
\end{equation}

\begin{equation} \label{deftigtithv} 
(\tilde g)_{\tith, w} := \sum_{(\ttv) \in \tilde \TTT_{\tith, w}} (\tilde g)_{\ttv}. 
\end{equation}

For any $g$, we get a decomposition $g = \sum_{\tith, w} g_{\tith, w}$ obeying with

\begin{equation} \label{orthgtithv} \| g \|_{L^2}^2 \sim \sum_{(\tith, w)} \| g_{\tith, w} \|_{L^2}^2, \end{equation}

Similarly, for any $\tg$, we get a decomposition $\tg = \sum_{\tith, w} \tg_{\tith, w}$ with

\begin{equation} \label{orthtigtithv} \| \tilde g \|_{L^2}^2 \sim \sum_{(\tith, w)} \| \tilde g_{\tith, w} \|_{L^2}^2. \end{equation}

By Lemma \ref{tubetildetube}, for all the pairs $(\theta, v) \in \TTT_{\tith, w}$, the sets $\Ttv \cap B(y, \rho)$ are essentially the same.  We denote this intersection by $T_{\tith, w} \subset B(y, \rho)$.  It is a tube of radius $R^{1/2 + \dt}$ and length $\rho$.  The set $\TTT_{\tith, w}$ can be described geometrically as the set of pairs $(\tv)$ so that $\Ttv \cap B(y, \rho)$ is essentially $ T_{\tith, w}$, and so that the direction of $\Ttv$ obeys the inequality $\Angle (G(\theta), G(\tilde \theta) ) \lesssim \rho^{-1/2}$.  

We will need to study the following situation.  We have a function $g$ which is concentrated on wave packets in $\TTT_Z$, and we want to study $\EP g$ on a smaller ball $B(y, \rho) \subset B_R$.  If we decompose $g$ into wave packets associated to the ball $B(y, \rho)$, what can we say about the new wave packet decomposition?

First of all, we point out the wave decompositon of $\tg$ at scale $\rho$ is not necessarily concentrated on wave packets that are tangent to $Z$ on $B(y, \rho)$.  By Lemma \ref{scalesftvtith}, we do know that $\tg$ is concentrated on wave packets in $\cup_{(\tv) \in \TTT_{Z}} \tilde \TTT_{\tv}$.  
If $(\tv) \in \TTT_Z$, then we know that $\Ttv$ is tangent to $Z$ on $B_R$, which implies that

\begin{equation*} \Ttv \cap B(y, \rho) \subset N_{R^{1/2 + \dm}}(Z) \cap B(y, \rho), \end{equation*}

and that for any $x \in \Ttv$ and $z \in Z \cap B(y, \rho)$ with $|x-z| \lesssim R^{1/2 + \dm}$, 

\begin{equation*} \Angle (G(\theta), T_z Z ) \lesssim R^{-1/2 + \dm}. \end{equation*}

If now $(\ttv) \in \tilde \TTT_{\tv}$, then (\ref{tubetildetubedist}) and (\ref{tubetildetubeangle}) imply that

\begin{equation*} \tTtv  \subset N_{R^{1/2 + \dm}}(Z) \cap B(y, \rho), \end{equation*}

and that for any $x \in \tTtv$ and $z \in Z \cap B(y, \rho)$ with $|x-z| \lesssim R^{1/2 + \dm}$, 

\begin{equation*} \Angle (G(\theta), T_z Z ) \lesssim R^{-1/2 + \dm} + \rho^{-1/2} \le \rho^{-1/2 + \dm}. \end{equation*}

The angle condition is more than strong enough for $\tTtv$ to be tangent to $Z$ in $B(y, \rho)$, but it is not true that $\tTtv \subset N_{\rho^{1/2 + \dm}} (Z) \cap B(y, \rho)$.  If $\tTtv$ intersects $N_{\rho^{1/2 + \dm}}(Z) \cap B(y, \rho)$, then the angle condition guarantees that $\tTtv$ is contained in $N_{2 \rho^{1/2 + \dm}} (Z) \cap B(y, \rho)$.  A bit more generally, if $b$ is a vector with $|b| \le R^{1/2 + \dm}$, and if $\tTtv$ intersects $N_{\rho^{1/2 + \dm}}(Z + b) \cap B(y,\rho)$, then $\tTtv$ is contained in $N_{2 \rho^{1/2 + \dm}} (Z+b) \cap B(y, \rho)$, and $\tTtv$ is tangent to $Z+b$ in $B_j$.

For any $b \in B_{R^{1/2 + \dm}}$, we define

$$ \tilde \TTT_{Z + b} := \{ (\ttv) : \tTtv \textrm{ is tangent to } Z+b \textrm{ in } B_j \}. $$

$$ \tg_{b} := \sum_{(\ttv) \in \tilde \TTT_{Z + b}} \tg_{\ttv}. $$

So we see that if $g$ is concentrated on wave packets in $\TTT_Z$, then $\tg$ is concentrated on wave packets in $\cup_{|b| \lesssim R^{1/2 + \dm}} \tilde \TTT_{Z+b}$.  For any $(\ttv) \in \cup_{(\tv) \in \TTT_Z} \tilde \TTT_{\tv}$, we saw above that either $(\ttv) \in \tilde \TTT_{Z+b}$, or else $\tTtv$ is disjoint from $N_{\rho^{1/2 + \dm}}(Z) \cap B(y, \rho)$.
Therefore, for $x = y + \tx \in B(y, \rho)$, 

\begin{equation} \label{Egb} | \EP \tg_b( \tx) | \sim  \chi_{N_{\rho^{1/2 + \dm}}(Z+b)}(x) \EP g(x). \end{equation}

To get finer information, it is helpful to decompose $g$ as above as $g = \sum_{\tith, w} g_{\tith, w}$, and to think about the wave packet decomposition of each piece $(g_{\tith, w})^\sim$ on $B(y, \rho)$.  For brevity, we let $h = g_{\tith, w}$.  We choose a ball $B(x_0, R^{1/2 + \dm})$ with $x_0 \in T_{\tith, w} \subset B(y, \rho)$.  For any $(\tv) \in \TTT_{\tith, w}$, $\Ttv \cap B(y, \rho) \subset T_{\tith, w}$, and so $\Ttv$ intersects $B(x_0, R^{1/2 + \dm})$ in a tube segment of length $R^{1/2 + \dm}$.  

By Lemma \ref{orthogloc2}, we have

\begin{equation} \label{hL2} \| h \|_{L^2}^2 \sim R^{-1/2 - \dm} \| \EP h \|_{L^2(B(x_0, R^{1/2 + \dm}))}^2. \end{equation}

Now we know that $ \EP \tih (\tx) = \EP h(\tx + y)$.  Also, for $x = y + \tx \in B(y, \rho)$, we know by (\ref{Egb}) that

$$ | \EP\tih_b (\tx) | \sim \chi_{N_{\rho^{1/2 + \dm}}(Z+b)}(x) \EP h(x). $$

Using Lemma \ref{orthogloc2} again, we have

\begin{equation} \label{hbL2} \| \tih_b \|_{L^2}^2 \sim R^{-1/2 - \dm} \| \EP h \|_{L^2(B(x_0, R^{1/2 + \dm}) \cap N_{\rho^{1/2 + \dm}}(Z+b))}^2. \end{equation}

These observations lead to a couple estimates about how $\| \tih_b \|_{L^2}$ relates to $\| h \|_{L^2}$.

\begin{lemma} \label{disjborthog} Suppose that $h$ is concentrated on wave packets in $\TTT_{\tith, w}$ for some $(\tith, w)$, and $x_0$ is in the tube $T_{\tith, w}$.  If we choose a set of vectors $b \in B_{R^{1/2 + \dm}}$ so that the sets 

$$B(x_0, R^{1/2 + \dm}) \cap N_{\rho^{1/2 + \dm}}(Z+b)$$

\noindent are disjoint, then

$$ \sum_b \| \tih_b \|_{L^2}^2 \lesssim \| h \|_{L^2}^2. $$

\end{lemma}

Finally, we come to transverse equidistribution estimates.  Combining the transverse equidistribution estimate in Lemma \ref{equid2} with the considerations in this section, we get the following estimates.

\begin{lemma} \label{equidgtithw}  Suppose that $h$ is concentrated on wave packets in $\TTT_Z$ and also on wave packets in $\TTT_{\tith, w}$ for some $(\tith, w)$.  Then for any $b \in B_{R^{1/2 + \dm}}$, we have

$$ \| \tilde h_b \|_{L^2}^2 \le R^{O(\dm)} \left( \frac{R^{1/2}}{\rho^{1/2}} \right)^{-(n-m)} \| h \|_{L^2}^2. $$

\end{lemma}

\begin{proof} We combine (\ref{hbL2}), Lemma \ref{equid2}, and (\ref{hL2}):

$$  \| \tilde h_b \|_{L^2}^2 \sim R^{-1/2 - \dm} \| \EP h \|_{L^2(B(x_0, R^{1/2 + \dm}) \cap N_{\rho^{1/2 + \dm}}(Z+b))}^2 $$

$$\lesssim R^{-1/2 - \dm} R^{O(\dm)}  \left( \frac{R^{1/2}}{\rho^{1/2}} \right)^{-(n-m)} \| \EP h \|_{L^2(B(x_0, R^{1/2 + \dm}))}^2 $$

$$ \sim R^{O(\dm)}  \left( \frac{R^{1/2}}{\rho^{1/2}} \right)^{-(n-m)} \| h \|_{L^2}^2. $$

\end{proof}

We can now combine the different $(\tith, w)$ in order to get estimates for $g$.

\begin{lemma} \label{equidsmallerball}
 If $g$ is concentrated in wave packets in $\TTT_Z$, then for any $b$,
$$ \| \tilde g_b \|_{L^2}^2 \le R^{O(\dm)} \left( \frac{R^{1/2}}{\rho^{1/2}} \right)^{-(n-m)} \| g \|_{L^2}^2. $$
\end{lemma}

\begin{proof}

We first expand $g = \sum_{\tith, w} g_{\tith, w}$.  The wave packets contributing significantly to each $g_{\tith, w}$ are a subset of those contributing to $g$, and so each $g_{\tith, w}$ is concentrated on wave packets in $\TTT_Z$.  For each $\tith, w$, Lemma \ref{equidgtithw} tells us that

$$ \| (g_{\tith, w})^\sim_b \|_{L^2}^2 \le R^{O(\dm)} \left( \frac{R^{1/2}}{\rho^{1/2}} \right)^{-(n-m)} \| g_{\tith,w} \|_{L^2}^2. $$

We know that the $g_{\tith, w}$ are orthogonal, and so

$$ \| g \|_{L^2}^2 \sim \sum_{\tith, w} \| g_{\tith, w} \|_{L^2}^2. $$

The operation $f \mapsto \tf_b$ is a linear map, and so

$$ \tg_b = \sum_{\tith, w} (g_{\tith, w})^\sim_b. $$

We claim that this is also an orthogonal decomposition.  By Lemma \ref{scalestithv}, $(g_{\tith, w})^\sim$ is concentrated on wave packets in $\tilde \TTT_{\tith, w}$.  But then $(g_{\tith, w})^\sim_b$ is also concentrated on wave packets in $\tilde \TTT_{\tith, w}$.  The different sets $\tilde \TTT_{\tith,w}$ are disjoint, and so the functions $(g_{\tith, w})^\sim_b$ are orthogonal, as claimed.  Therefore

$$ \| \tg_b \|_{L^2}^2 \sim \sum_{\tith, w} \| (g_{\tith, w})^\sim_b \|_{L^2}^2. $$

Combining these estimates gives the desired conclusion.  \end{proof}

\section{Proof of Theorem \ref{broadEP}.    }  \label{secmainind}

We now formulate the inductive estimate that proves Theorem \ref{broadEP}.

\begin{prop} \label{mainind} For $\eps > 0$, there are small constants $ 0 < \delta \ll \delta_{n-1} \ll ... \ll \delta_1 \ll \delta_0 \ll \eps,$ and a large constant $\bar A$ so that the following holds.

Let $m$ be a dimension in the range $1 \le m \le n$.  Suppose that $Z = Z(P_1, ..., P_{n-m})$ is a transverse complete intersection where $\Deg P_i \le D_Z$.  Suppose that $f$ is concentrated on wave packets from $\TTT_Z$.  Then for any $1 \le A \le \bar A$, and any radius $1 \le R$, 

\begin{equation} \label{mainindest} \| \EP f \|_{\BL^p_{k, A} (B_R)} \le C(K, \eps, m, D_Z) R^{m \eps} R^{\delta (\log \bar A - \log A)} R^{ - e + \frac{1}{2}}  \| f \|_{L^2}, \end{equation}

for all 

\begin{equation} \label{rangep} 2 \le p \le \pkm := 2 \cdot \frac{m+k}{m+k-2}, \end{equation}

where

\begin{equation} \label{defe} e = e(k,n,p) = \frac{1}{2} \left( \frac{1}{2} - \frac{1}{p} \right) (n+k). \end{equation}

\end{prop}

When $m=n$, Proposition \ref{mainind} gives Theorem \ref{broadEP}.   When $m= n$, we can take $Z = \RR^n$ (and $D_Z =1$).  Now if we choose $A=\bar A$ and $p = \pkn$, then we compute $-e + 1/2 = 0$, and we get the inequality in Theorem \ref{broadEP}.

We prove Proposition \ref{mainind} by induction.  We will do induction on the dimension $m$, the radius $R$, and on $A$.   We start by checking the base of the induction.  When $R$ is small, we choose the constant $C(K, \eps, m, D_Z)$ sufficiently large and the result follows.  So from now on, we can assume that $R$ is very large compared to $K, \eps, m, D_Z$.  To check the case $A =1$, we choose $\bar A$ large enough so that $R^{\delta( \log \bar A - \log 1 )} = R^{10 n}$, and the inequality follows because $ \| \EP f \|_{\BL^p_{k, 1} (B_R)} \le | B_R | \| f \|_{L^2}$.  The base of the induction on $m$ is $m = k-1$.  In this case, since $A \ge 1$, we have

\begin{equation} \label{m=k-1} \| \EP f \|_{\BLpkA (B_R)} \le \RD(R) \| f \|_{L^2}. \end{equation}

\noindent This follows from the definition of $\BLpkA$.  Recall that $ \| \EP f \|_{\BLpkA(B_R)}^p := \sum_{B_{K^2} \subset B_R} \mu_{Ef}(B_{K^2})$, where 

$$ \mu_{\EP f} (B_{K^2}) := \min_{V_1, ..., V_A \textrm{ $(k-1)$-subspaces of } \RR^n} \left( \max_{\tau: \Angle(G(\tau), V_a) > K^{-1} \textrm{ for all } a } \int_{B_{K^2}} | \EP f_\tau|^p \right). $$

\noindent  Fix a ball $B = B_{K^2} \subset \NZR$, and let $V$ be the tangent space to $Z$ at some point $z$ in the $R^{1/2 + \dm}$-neighborhood of the ball $B_{K^2}$.   Notice that the dimension of $V$ is $m = k-1$.  If $\Ttv$ intersects $B_{K^2}$ and if $(\tv) \in \TTT_Z$, then $\Angle (G(\theta), V) \lesssim R^{-1/2 + \dm}$.  So if $\tau$ contains a $\theta$ with $(\tv) \in \TTT_Z$ for some $v$, then $\Angle( G(\tau), V) \le K^{-1}$.  On the other hand, if $\tau$ does not contain such a $\theta$, then $\| f_\tau \|_{L^2} \le \RD(R) \| f \|_{L^2}$, because $f$ is concentrated on wave packets tangent to $Z$.  Therefore,

$$ \left( \max_{\tau: \Angle(G(\tau), V) > K^{-1} } \int_{B_{K^2}} | \EP f_\tau|^p \right) \le \RD(R) \| f \|_{L^2}^p. $$

\noindent On the other hand, if $B_{K^2} \subset B_R$ is not contained in $N_{R^{1/2 + \dm}}(Z)$, then on $B_{K^2}$, $| \EP f | \le \RD(R) \| f \|_{L^2}$.  This proves (\ref{m=k-1}), establishing the base case $m=k-1$.

For $p=2$, Proposition \ref{mainind} follows quickly from the basic $L^2$ estimate in Lemma \ref{basicL2}:

$$ \| \EP f \|_{\BL^2_{k, A}(B_R)}^2 \le \sum_\tau \| \EP f_\tau \|_{L^2(B_R)}^2 $$

$$ \lesssim R \sum_\tau \| f_\tau \|_{L^2}^2 = R \| f \|_{L^2}^2. $$

Next we begin the inductive step.  We assume that Proposition \ref{mainind} holds if we decrease the dimension $m$, the radius $R$, or the value of $A$.  We define $p$ to be 

$$ p = \bar p(k,m) \textrm{ if } m > k, $$

$$ p = \bar p(m, m) + \delta \textrm{ if } m =k. $$

We will check the main estimate (\ref{mainindest}) for this value of $p$.  Once we check (\ref{mainindest}) for this value of $p$, we can get the whole range in (\ref{rangep}) by interpolation between 2 and our value of $p$, using the Holder inequality for $\BL$ in Lemma \ref{broadholder}.

There are two cases, depending on whether or not the mass of $\mu$ is concentrated into a small neighborhood of a lower dimensional variety.  We let $D(\eps, D_Z)$ be a function that we will define later.  We say we are in the algebraic case if there is a transverse complete intersection $Y^l \subset Z^m$ of dimension $l < m$, defined using polynomials of degree $\le D(\eps, D_Z)$, so that 

$$ \mu_{\EP f}( \NY \cap B_R) \gtrsim \mu_{\EP f}(B_R). $$

\noindent Otherwise, we say that we are in the non-algebraic case, or the cellular case.

\subsection{The non-algebraic case}  We begin with the non-algebraic case.  In this case, we will use polynomial partitioning.  To set up the polynomial partitioning, we first locate a significant piece of $\NZR$ where the tangent space of $Z$ is not changing too fast.  We say that a ball $B(x_0, R^{1/2 + \dm}) \subset \NZR$ is regular if, on each connected component of $Z \cap B(x_0, 10 R^{1/2 + \dm})$, the tangent space $TZ$ is constant up to angle $\frac{1}{100}$.  Let $w \in \Lambda^m \RR^n$.  Recall that $Z_w \subset Z$ is defined in (\ref{defZw}).  For generic $w$, $Z_w \subset Y$ is a transverse complete intersection of dimension $m-1$, defined using polynomials of degree $\lesssim D_Z$.  We can choose a set of $\lesssim 1$ values of $w$ so that on each connected component of $Z \setminus \cup_w Z_w$, the tangent plane $TZ$ is constant up to angle $\frac{1}{100}$.  Since we are in the non-algebraic case,

$$ \mu_{\EP f} ( \cup_w N_{10 R^{1/2 + \dm}}(Z_w) \cap B_R ) \le \frac{1}{100} \mu_{\EP f} (B_R). $$

\noindent Each ball $B(x_0, R^{1/2 + \dm}) \subset \NZR$ which does not intersect $\cup_w N_{10 R^{1/2 + \dm}}(Z_w)$ is a regular ball.  So the regular balls contain most of the mass of $\mu$.

For each regular ball $B = B_{R^{1/2 + \dm}} \subset \tNZR$, we pick a point $z \in Z \cap B_{R^{1/2 + \dm}}$ and we define $V_B$ to be the $m$-plane $T_z Z$.  For an $m$-plane $V$, we define $ \frak B_V$ to be the set of regular balls so that $\Angle(V_B, V) \le \frac{1}{100}$.  By pigeonholing, we can choose a plane $V$ so that

$$ \mu_{\EP f} \left( \cup_{B \in \frak B_V} B \right) \gtrsim \mu_{\EP f} (B_R). $$

\noindent We define $N_1 \subset \NZR$ to be the union of the balls $B \in \frak B_V$.  We let $\mu_1$ be the restriction of $\mu_{\EP f}$ to $N_1$, and we note that $\mu_1 (N_1) \sim \mu_{\EP f} (B_R)$.  

Now we are ready to do polynomial partitioning.  We let $P_V$ denote a polynomial defined on $V$: $P_V: V \rightarrow \RR$.  
We let $\pi: \RR^n \rightarrow V$ be the orthogonal projection.  Now we apply polynomial partitioning, Theorem \ref{polypart}, to the push-forward measure $\pi_* \mu_1$ on $V$, using the degree $D = D(\eps, D_Z)$.  Theorem \ref{polypart} gives us a non-zero polynomial $P_V$ of degree at most $D$ so that $V \setminus Z(P_V) = \cup_i O_{V, i}$, the number of cells $O_{V, i}$ is $\sim D^m$, and for each cell, $\pi_* \mu_1 (O_{V,i})$ is $\sim D^{-m} \pi_* \mu_1 (V)$.  Moreover, $P_V = \prod_j Q_{V,j}$ where each $Q_{V,j}$ has a little freedom in the constant term, which we can use for transversality purposes.  

We extend $V$ to a polynomial $P$ on $\RR^n$ by setting $P(x) := P_V( \pi (x) )$.  We note that $Z(P) = \pi^{-1} ( Z(P_V) )$.  We define $O_i := \pi^{-1} O_{V,i}$, and we note that $\RR^n \setminus Z(P) = \cup_i O_i$ and that $\mu_1 (O_i) = \pi_* \mu_1 ( O_{V,i}) \sim D^{-m} \mu_1( N_1)$.  Similarly, we define $Q_j(x) = Q_{V,j} (\pi(x) )$ so that $P = \prod_j Q_j$.  Each $Q_j$ is a polynomial of degree at most $D$ on $\RR^n$, and we have a little freedom in the constant term of each $Q_j$.  By Lemma \ref{sardtci}, we can guarantee that for each $j$, $Y_j = Z(P_1, ..., P_{n-m}, Q_j)$ is a transverse complete intersection.  

We define $W := N_{R^{1/2 + \dt}} Z(P)$, and $O_i' := O_i \setminus W$.  Since each tube $\Ttv$ has radius $R^{1/2 + \dt}$, each tube $\Ttv$ enters at most $D$ of the cells $O_i'$.  On the other hand, we claim that 

$$ W \cap N_1 \subset \bigcup_j N_{20 R^{1/2 + \dm}} (Y_j) . $$

\noindent Here is the proof of the claim.  Suppose that $x \in W \cap N_1$.  Since $x \in N_1$, $x$ is in a regular ball $B = B(x_0, R^{1/2 + \dm}) \in \frak B_V$.  There is a point $z_B \in Z \cap B$ with $\Angle(T_{z_B} Z, V) \le \frac{1}{100}$.  Let $Z_B$ be the component of $Z \cap 10 B$ containing this point $z_B$.  On $Z_B$, the tangent plane $TZ$ makes a small angle with $V$.  
Let us use coordinates $(v,w)$, where $v \in V$ and $w \in V^\perp$.  Since the tangent plane of $TZ_B$ makes a small angle with $V$, $Z_B$ is the graph of a function, $w = h(v)$, where $h$ has small Lipschitz constant $\le \frac{1}{100}$.  Since $x \in W$, there must be a point $(v_0, w_0) \in Z(P) \cap B(x, R^{1/2 + \dt})$.  Since $P(v,w) = P_V(v)$ we see that $P_V(v_0) = 0$.  Now the point $(v_0, h(v_0))$ lies both in $Z(P)$ and in $Z_B$, and so it lies in $Y_j$ for some $j$.  Since $(v_0, h(v_0)) \in Z_B \subset B(x_0, 10 R^{1/2 + \dm})$ and $x \in B(x_0, R^{1/2 + \dm})$, it follows that $x \in N_{20 R^{1/2 + \dm}}(Y_j)$ as desired.  

Since we are in the non-algebraic case, $\mu_1 (N_{20 R^{1/2 + \dm}} (Y_j) )$ is negligible for each $j$, and so $\mu_1 (W) = \mu_1 (W \cap N_1)$ is negligible.  Therefore, we have

\begin{equation} \label{equiO_i'} \mu_1( O_i') \sim D^{-m} \mu_1(N_1) \sim D^{-m} \| \EP f \|_{\BLpkA(B_R)}^p \textrm{ for the vast majority of } i. \end{equation}

Now for each index $i$ we define a function $f_i$ which only includes the wave packets that enter $O_i'$.  More precisely, $f_i = \sum_{(\tv) \in \TTT_i} \ftv$ where

\begin{equation} \label{defT_i} 
\TTT_i := \{ (\tv) \textrm{ such that } \Ttv \cap O_i' \not= \emptyset \}. 
\end{equation}

Each function $f_i$ is also concentrated on wave packets in $\TTT_Z$.  Moreover, there are $\sim D^m$ cells $O_i'$ for which

$$ \| \EP f \|_{\BLpkA(B_R)}^p \lesssim D^m \mu_{1} ( O_i' ) \lesssim D^m \mu_{\EP f_i} (O_i' ) \lesssim D^m \| \EP f_i \|_{\BLpkA(B_R)}^p. $$

On the other hand, we have good control on the $L^2$ norms of $f_i$.  Because each tube $\Ttv$ enters $\lesssim D$ cells $O_i'$, each pair $(\tv)$ belongs to $\TTT_i$ for $\lesssim D$ values of $i$, and so

$$ \sum_i \| f_i \|_{L^2}^2 \sim \sum_i \sum_{(\tv) \in \TTT_i} \| \ftv \|_{L^2}^2 \lesssim D \sum_{(\tv)} \| \ftv \|_{L^2}^2 \sim D \| f \|_{L^2}^2. $$

In summary, there are $\sim D^m$ choices of $i$ so that

\begin{equation} \label{Brf_i} 
 \| \EP f \|_{\BLpkA(B_R)}^p \lesssim D^m \| \EP f_i \|_{\BLpkA(B_R)}^p. 
 \end{equation}
 
\begin{equation} \label{L^2f_i}
 \| f_i \|_{L^2}^2  \lesssim D^{1-m} \| f \|_{L^2}^2.
\end{equation}

(For later reference, we also record here: for each $i$ and each $\theta$, 

\begin{equation} \label{L^2f_itheta}
\| f_{i, \theta} \|_{L^2}^2 \lesssim \sum_{\theta' \cap \theta \not= \emptyset} \| f_{\theta'} \|_{L^2}^2.
\end{equation}

\noindent This inequality doesn't appear in the proof of Theorem \ref{broadEP}, but it could be useful in some later refinements.)

Using these estimates, we can now prove (\ref{mainindest}) by induction on the radius.  To make the computation clearer, we abbreviate

\begin{equation*} E = \bar C(K,  \eps, m, D_Z) R^{m \eps} R^{\delta (\log \bar A - \log A)} R^{ - e + \frac{1}{2}}, \end{equation*}

so (\ref{mainindest}) reduces to

$$ \| \EP f \|_{\BLpkA(B_R)} \le E \| f \|_{L^2}. $$

Since we assume that (\ref{mainindest}) holds on balls of radius $R/2$, it follows that it also holds on balls of radius $R$ up to a constant factor loss.  So we can assume that

\begin{equation} \| \EP f_i \|_{\BL^p_{k, A} (B_R)} \le C E \| f_i \|_{L^2}. \end{equation}

Now using (\ref{Brf_i}) and (\ref{L^2f_i}), we get

$$  \| \EP f \|_{\BLpkA(B_R)}^p \lesssim D^m \| \EP f_i \|_{\BLpkA(B_R)}^p \lesssim D^m E^p \| f_i \|_{L^2}^p \lesssim D^{m - (m-1)\frac{p}{2}} E^p \| f \|_{L^2}^p. $$

By our choice of $p$, $p > \frac{2m}{m-1}$, and so the exponent of $D$ is negative.  If $m > k$, the exponent is negative with pretty large absolute value; if $m = k$, then the exponent is $- \delta$.  
If we choose $D = D(\eps, D_Z)$ sufficiently large, the power of $D$ dominates the implicit constant, and we get $\| \EP f \|_{\BLpkA(B_R)} \le E \| f \|_{L^2}$, which closes the induction in the non-algebraic case.

\subsection{The algebraic case} Next we turn to the algebraic case.  By definition, we know that there is a transverse complete intersection $Y^l$ of dimension $l < m$, defined using polynomials of degree $\le D = D(\eps, D_Z)$, so that 

\begin{equation} \label{Ybig} \mu_{\EP f}( \NY) \gtrsim \mu_{\EP f}(B_R). \end{equation}

In the algebraic case, we subdivide $B_R$ into smaller balls $B_j$ of radius $\rho$, chosen so that

\begin{equation} \label{defrho} \rho^{1/2 + \dl} = R^{1/2 + \dm}. \end{equation}


For each $j$, we define $f_j = \sum_{(\tv) \in \TTT_j} \ftv$, where

$$ \TTT_j := \{ (\tv) | \Ttv \cap N_{R^{1/2 + \dm}}(Y) \cap B_j \not= \emptyset \}. $$

On $N_{R^{1/2 + \dm}}(Y) \cap B_j$, $\EP f_j = \EP f + \RD(R) \| f \|_{L^2}$.  Therefore,

$$ \| \EP f \|_{\BLpkA(B_R)}^p \lesssim \sum_j \| \EP f_j \|_{\BLpkA(B_j)}^p + \RD(R) \| f \|_{L^2}^p. $$

We further subdivide $\TTT_j$ into tubes that are tangential to $Y$ and tubes that are transverse to $Y$.  As in definition \ref{defTZ}, we say that $\Ttv$ is tangent to $Y$ in $B_j$ if 

\begin{equation} \label{tangjdist} \Ttv \cap 2 B_j \subset N_{R^{1/2 + \dm}}(Y) \cap 2 B_j = N_{\rho^{1/2 + \dl}}(Y) \cap 2 B_j, \end{equation}

and for any $x \in \Ttv$ and $y \in Y \cap 2 B_j $ with $|x-y| \lesssim R^{1/2 + \dm} = \rho^{1/2 + \dl}$, 

\begin{equation} \label{tangjangle} \Angle (G(\theta), T Y_y ) \lesssim \rho^{-1/2 + \dl}. \end{equation}

We define the tangential wave packets by

$$ \TTT_{j, tang} = \{ (\tv) \in \TTT_j | \Ttv \textrm{ is tangent to $Y$ in $B_j$} \}. $$

And we define the transverse wave packets by 

$$ \TTT_{j, trans} := \TTT_j \setminus \TTT_{j, trans}. $$

We define $f_{j, tang} = \sum_{(\tv) \in \TTT_{j, tang}} \ftv$ and $f_{j, trans} = \sum_{(\tv) \in \TTT_{j, trans} \ftv}$ so 

$$ f_j = f_{j, tang} + f_{j, trans}. $$

Therefore we have

\begin{equation} \label{tang+trans} \sum_j \| \EP f_j \|_{\BLpkA(B_j)}^p  \lesssim  \sum_j \| \EP f_{j, tang} \|_{\BLpkAt(B_j)}^p + \sum_j \| \EP f_{j, trans} \|_{\BLpkAt(B_j)}^p. \end{equation}

We will control the contribution of the tangential wave packets by induction on the dimension $m$, and we will control the contribution of the transverse wave packets by induction on the radius $R$.

\subsection{The tangential sub-case} Suppose first that the tangential wave packets dominate the right-hand side of (\ref{tang+trans}).  In order to apply induction to $\EP f_{j, tang}$ on $B_j$, we redo the wave packet decomposition at a scale appropriate to $B_j$, as in Section \ref{secsmallerball}.  For brevity, during this discussion, we let $g = f_{j, tang}$.

$$ \tg = \sum_{\ttv} \tg_{\ttv} + \RD(R) \| f \|_{L^2}. $$

Before applying induction, we need to check that this wave packet decomposition is concentrated on pairs $(\ttv)$ that are tangent to $Y$ on $B_j$, in the sense of Definition \ref{defTZ}: in other words on pairs $(\ttv)$ so that

\begin{equation} \label{tangYdist}  \tTtv \subset N_{\rho^{1/2 + \dl}}(Y) \cap B_j, \end{equation}

\noindent and for any $x \in \tTtv$ and $y \in Y \cap B_j$ with $|x-y| \lesssim \rho^{1/2 + \dl}$, 

\begin{equation} \label{tangYangle}  \Angle (G(\tith), T_y Y ) \lesssim \rho^{-1/2 + \dl}. \end{equation}

We know that $g = f_{j, tang}$ is concentrated on wave packets $(\tv) \in \TTT_{j, tang}$, which obey (\ref{tangjdist}) and (\ref{tangjangle}).  
These tell us that $\Ttv \cap B_j$ lies in the desired neighborhood of $Y \cap B_j$ and makes a good angle with $T_y Y$.  For any $(\tv)$, $(\ftv)^\sim$ is concentrated on wave packets $(\ttv) \in \tilde \TTT_{\tv}$, by Lemma \ref{scalesftvtith}.  For $(\ttv) \in \tilde \TTT_{\tv}$, (\ref{tubetildetubedist}) and (\ref{tubetildetubeangle}) tell us that $\tTtv$ lies in a small neighborhood of $\Ttv \cap B_j$ and makes a small angle with $\Ttv$.  So if $(\tv) \in \TTT_{j, tang}$ and $(\ttv) \in \tilde \TTT_{\tv}$, then $\tTtv$ obeys (\ref{tangYdist}) and (\ref{tangYangle}).  We have now checked the hypotheses of Proposition \ref{mainind} for $\tg$ with the variety $Y$ on the ball $B_j$, and so we can apply induction on the dimension.  

By induction on the dimension, we get the following inequality:

$$ \| \EP f_{j, tang} \|_{\BLpkAt(B_j)} = \| \EP \tg \|_{\BLpkAt(B_\rho)} $$

$$ \le C(K, \eps, l, D(\eps, D_Z) )  \rho^{l \eps} \rho^{\delta (\log \bar A - \log (A/2) )} \rho^{ - e + \frac{1}{2}}  \| f_{j, tang} \|_{L^2}, $$

for 

$$2 \le p \le \bar p(k,l) := 2 \cdot \frac{l+k}{l+k-2},$$

where

$$ e = e(k,n,p) = \frac{1}{2} \left( \frac{1}{2} - \frac{1}{p} \right) (n+k). $$

Since $l < m$, $\bar p(k,m) < \bar p(k,l)$, and so our value of $p$ is in the range $2 \le p \le \bar p(k,l)$ and the bound above applies.  The number of balls $B_j$ is $\lesssim R^{O(\dl)}$.  Summing brutally over the balls, we see that

$$ \| \EP f \|_{\BLpkA(B_R) } \lesssim R^{O(\dl)} C(K, \eps, l, D(\eps, D_Z)) \rho^{l \eps} \rho^{\delta (\log \bar A - \log (A/2))} \rho^{ - e + \frac{1}{2}} \| f \|_{L^2}. $$

We note that the exponent $-e + 1/2$ may well be negative.  Nevertheless, $\rho^{-e + 1/2} \le R^{O(\dl)} R^{-e + 1/2}$.  Also, $\rho^{\delta (\log \bar A - \log A/2)} \le R^{\delta (\log \bar A - \log A/2)} = R^\delta R^{\delta (\log \bar A - \log A)}$.  Therefore, the last expression is

$$ \le R^{O(\dl)} C(K, \eps, l, D(\eps, D_Z)) R^{l \eps} R^{\delta (\log \bar A - \log (A/2))} R^{ - e + \frac{1}{2}} \| f \|_{L^2}. $$

Since $\dl$ is much smaller than $\eps$, $R^{O(\dl)} R^{l \eps} \le R^{m \eps}$, and the induction closes.  (We have to choose $C(K, \eps, m, D_Z)$ larger than $C(K, \eps, l, D(\eps, D_Z))$. )  This finishes the discussion of the tangential algebraic case.  

\subsection{The transverse sub-case} Suppose now that the transverse wave packets dominate (\ref{tang+trans}).  First, we note that

$$ \sum_j \| f_{j, trans} \|_{L^2}^2 = \sum_{\tv} |\{ j: (\tv) \in \TTT_{j, trans} \}| \| \ftv \|_{L^2}^2. $$

Next we claim that for each $(\tv)$, $|\{ j: (\tv) \in \TTT_{j, trans} \}| \lesssim_{\eps, D_Z} 1$.  In the discussion, we just abbreviate this as $\lesssim 1$.  This follows from Lemma \ref{transinterbound}, which controls the transverse intersections between a tube and an algebraic variety.  Let $T$ be the cylinder with the same center as $\tTtv$ and with radius $r = R^{1/2 + \dm} = \rho^{1/2 + \dl}$, and let $\alpha = \rho^{- 1/2 + \dl}$.  Let $\ell$ denote the central axis of $T$ and recall that $Y_{> \alpha}$ is the set $\{ y \in Y | \Angle ( T_y Y, \ell) > \alpha \}$.  If $(\tv) \in \TTT_{j, trans}$, then $T \cap Y_{> \alpha} \cap 2 B_j$ must be non-empty.  However, Lemma \ref{transinterbound} tells us that $T \cap Y_{> \alpha}$ is contained in $\le C D_Y^n$ balls of radius $\lesssim r \alpha^{-1} \sim \rho$.  Here $Y$ is defined by polynomials of degree $D_Y \le D(\eps, D_Z) \lesssim 1$.  Therefore, $(\tv) \in \TTT_{j, trans}$ for $\lesssim 1$ values of $j$.  Plugging this into the last equation, we get

\begin{equation} \label{fjtransl2}  \sum_j \| f_{j, trans} \|_{L^2}^2 \lesssim \| f \|_{L^2}^2. \end{equation}

Next we would like to study $\EP f_{j,trans}$ on each ball $B_j$ by doing induction on the radius.  
In order to do so, we redo the wave packet decomposition at a scale appropriate to $B_j$, as in Section \ref{secsmallerball}.  For brevity, during this discussion, we let $g = f_{j, trans}$.  

$$ \tg = \sum_{\ttv} \tg_{\ttv} + \RD(R) \| f \|_{L^2}. $$

We recall a couple definitions from Section \ref{secsmallerball}.  For any $b \in B_{R^{1/2 + \dm}}$, we define

$$ \tilde \TTT_{Z + b} := \{ (\ttv) : \tTtv \textrm{ is tangent to } Z+b \textrm{ in } B_j \}. $$

$$ \tg_{b} := \sum_{(\ttv) \in \tilde \TTT_{Z + b}} \tg_{\ttv}. $$

\noindent For each $b$, $\tg_{b}$ is concentrated in wave packets tangent to $Z + b$ in the ball $B_j$, and so we will be able to apply induction on the radius to study $\EP \tg_b$.  By (\ref{Egb}), if $y_j$ is the center of $B_j$ and $x = y_j + \tx \in B_j$, then

\begin{equation} \label{Egb2} | \EP \tg_b( \tx) | \sim  \chi_{N_{\rho^{1/2 + \dm}}(Z+b)}(x) | \EP g(x) |. \end{equation}

We define $f_{j, trans, b}$ so that $(f_{j, trans, b})^\sim = \tg_b$ (in other words, $f_{j, trans, b} = e^{- i \psi_y(\omega)} \tg_b$).  In this language, the last equation becomes

\begin{equation} \label{Efjb} | \EP f_{j, trans, b} (x) | \sim \chi_{N_{\rho^{1/2 + \dm}}(Z+b)} (x) |\EP f_{j, trans}(x)|. \end{equation}

Next we choose a set of vectors $b \in B_{R^{1/2 + \dm}}$.  The number of vectors $b$ that we choose is related to the geometry of $Z$.  We cover $N_{R^{1/2 + \dm}}(Z) \cap B_j$ with disjoint balls of radius $R^{1/2 + \dm}$, and in each ball $B$ we note the volume of $B \cap N_{\rho^{1/2 + \dm}}(Z)$.  We dyadically pigeonhole this volume: for each $s$ we consider

$$ \frak B_s := \{ B(x_0, R^{1/2 + \dm}) \subset N_{R^{1/2 + \dm}}(Z) \cap B_j : | B(x_0, R^{1/2 + \dm}) \cap N_{\rho^{1/2 + \dm}}(Z) | \sim 2^s \}. $$

\noindent We select a value of $s$ so that 

$$ \| \EP f_{j, trans} \|_{\BLpkAt(\cup_{B \in \frak B_s} B)} \gtrsim (\log R)^{-1}  \| \EP f_{j, trans} \|_{\BLpkAt(B_j)}. $$

Next we prune $\TTT_{j, trans}$ a little: we include $(\theta, v)$ only if $\Ttv$ intersects one of the balls of $\frak B_s$.  To avoid making the notation even heavier, we don't make a separate notation for the pruned set.  This pruning can only decrease $\| f_{j, trans} \|_{L^2}$, and it changes $\| \EP f_{j, trans} \|_{\BLpkAt(B_j)}^p$ by at most a factor of $\log R$.  

Now we are ready to choose our set of translations $\{ b \}$.  We choose a random set of $ | B_{R^{1/2 + \dm}} | / 2^s$
vectors $b \in B_{R^{1/2 + \dm}}$.  For a typical ball $B(x_0, R^{1/2 + \dm}) \in \frak B_s$, the union $\cup_b N_{\rho^{1/2 + \dm}}(Z+b)$ covers a definite fraction of the ball (in a random way).   Therefore, with high probability, we get

\begin{equation} \label{breakupjb} \| \EP f_{j, trans} \|_{\BLpkAt(B_j)}^p \lesssim \sum_b \| \EP f_{j, trans, b} \|_{\BLpkAt(N_{\rho^{1/2 + \dm}}(Z + b) \cap B_j)}^p. \end{equation}

\noindent On the other hand, a typical point of $B(x_0, R^{1/2 + \dm})$ lies in $\lesssim 1$ of the sets $N_{\rho^{1/2 + \dm}}(Z+b)$.  Using this geometric fact, we will show that

\begin{equation} \label{breakupjborthog} \sum_b \| \tg_{b} \|_{L^2}^2 \lesssim \| g \|_{L^2}^2. \end{equation}

To see (\ref{breakupjborthog}), we decompose $g = \sum_{\tith, w} g_{\tith, w}$ as in Section \ref{secsmallerball}.  If $g_{\tith, w}$ is not neglibigle, then $T_{\tith, w}$ must intersect one of the balls $B(x_0, R^{1/2 + \dm}) \in \frak B_s$.  Since the sets $N_{\rho^{1/2 + \dm}}(Z+b) \cap B(x_0, R^{1/2 + \dm})$ are essentially disjoint, Lemma \ref{disjborthog} tells us that

$$ \sum_b \| (g_{\tith, w})^\sim_b \|_{L^2}^2 \lesssim \| g_{\tith, w} \|_{L^2}^2. $$

But (as we saw in the proof of Lemma \ref{equidsmallerball}), 

$$ \tg_b = \sum_{\tith, w} (g_{\tith,w})^\sim_b $$

\noindent is an orthogonal decomposition, and $g = \sum_{\tith, w} g_{\tith, w}$ is an orthogonal decomposition, and so

$$ \sum_b \| \tg_b \|_{L^2}^2 \sim \sum_{b, \tith, w} \| (g_{\tith, w})^\sim_b \|_{L^2}^2 \lesssim \sum_{\tith, w} \| g_{\tith, w} \|_{L^2}^2 \sim \| g \|_{L^2}^2. $$ 

We now have all the estimates that we need in the transverse case, and we collect them here. 

\begin{equation} \label{jbdecBrLp}
 \| \EP f \|_{\BLpkA(B_R) }^p \lesssim \log R \sum_{j,b} \| \EP f_{j, trans, b} \|_{\BLpkAt(B_j)}^p. 
 \end{equation}

\begin{equation} \label{jbdecorth}
\sum_{j,b} \| f_{j, trans, b} \|_{L^2}^2 \lesssim \| f \|_{L^2}^2.
\end{equation}

By Lemma \ref{equidsmallerball}, 

\begin{equation} \label{jbdecequi}
\max_b \| f_{j, trans, b} \|_{L^2}^2  \le R^{O(\dm)} \left( \frac{R^{1/2}}{\rho^{1/2}} \right)^{-(n-m)} \| f_{j, trans} \|_{L^2}^2.
\end{equation}

(For later reference, we also record here: for each $j,b$ and each $\tilde \theta$, 

\begin{equation} \label{jbdecequitith}
\| f_{j, trans, b} \|_{L^2(\tith)}^2 \lesssim R^{O(\dm)} \left( \frac{R^{1/2}}{\rho^{1/2}} \right)^{-(n-m)} \| f \|_{L^2(2 \tith)}^2.
\end{equation}

\noindent This follows from Lemma \ref{equidsmallerball}.  This inequality doesn't appear in the proof of Theorem \ref{broadEP}, but it could be useful in some later refinements.)

By induction on the radius, we know that

\begin{equation*} \| \EP f_{j, trans,b} \|_{\BL^p_{k, A/2} (B_j)} \le C(K, \eps, m, D_Z) \rho^{m \eps}  \rho^{\delta (\log \bar A - \log (A/2))} \rho^{ - e + \frac{1}{2}}  \| f_{j, trans, b} \|_{L^2} \end{equation*}

$$ \le C(K, \eps, m, D_Z) R^\delta \rho^{m \eps} R^{\delta (\log \bar A - \log A)} \rho^{ - e + \frac{1}{2}}  \| f_{j, trans, b} \|_{L^2} . $$

Using these estimates, we can now prove (\ref{mainindest}) by induction on the radius.  

$$ \| \EP f \|_{\BLpkA(B_R)}^p \lesssim (\log R) \sum_{j,b} \| \EP f_{j, trans, b} \|_{\BLpkAt(B_j)}^p $$

$$ \le R^{O(\delta)} \left( C(K, \eps, m, D_Z) \rho^{m \eps}  R^{\delta (\log \bar A - \log (A/2))} \rho^{ - e + \frac{1}{2}} \right)^p \sum_{j,b}  \| f_{j, trans, b} \|_{L^2}^p. $$

Using (\ref{jbdecequi}) and (\ref{jbdecorth}),

$$ \sum_{j,b}  \| f_{j, trans, b} \|_{L^2}^p \lesssim R^{O(\dm)}  \left( \frac{R^{1/2}}{\rho^{1/2}} \right)^{-(n-m)(\frac{p}{2}-1)} \| f \|_{L^2}^p. $$

Now plugging in we get

$$ \| \EP f \|_{\BLpkA(B_R)}^p \lesssim R^{O(\dm)} \left( C(K, \eps, m, D_Z) \rho^{m \eps} R^{\delta (\log \bar A - \log A)} \rho^{ - e + \frac{1}{2}} \right)^p \left( \frac{R^{1/2}}{\rho^{1/2}} \right)^{-(n-m)(\frac{p}{2}-1)} \| f \|_{L^2}^p. $$

At the exponent $p = \bar p(k,m)$, 

$$ \left( \rho^{-e + 1/2} \right)^{p}  \left( \frac{R^{1/2}}{\rho^{1/2}} \right)^{-(n-m)(\frac{p}{2}-1)} = \left( R^{-e + 1/2} \right)^{p}. $$

(If $ m =k$, so that $p = \bar p(k,m) + \delta$, then this is true up to a factor $R^{O(\delta)}$.)
So plugging in our values of $p$ and $e$, and multiplying out, we get

$$  \| \EP f \|_{\BLpkA(B_R)}^p \le C R^{O(\dm)} (R/\rho)^{- m \eps}  \left( C(K, \eps, m, D_Z) R^{m \eps} R^{\delta (\log \bar A - \log A)} R^{ - e + \frac{1}{2}} \right)^p  \| f \|_{L^2}^p. $$

The constant $C$ on the right-hand side depends on $\eps$, $D_Z$ and the dimension $n$.
We have to check that $C R^{O(\dm)} (R/ \rho)^{-m \eps} \le 1$.  Note that $R/\rho = R^{\theta(\dl)}$.  We choose the $\delta$'s so that $\delta_m \ll \eps \delta_{m-1}$, and so $(R/\rho)^{- m \eps}$ dominates the other terms.  Therefore, the induction closes in the transverse algebraic case also.  

This finishes the proof of Theorem \ref{broadEP}.

\section{Going from $k$-broad estimates to regular estimates}

The paper \cite{BG} introduces a technique to go from multilinear estimates to regular $L^p$ estimates.  In this section, we follow this technique to go from $k$-broad estimates to regular $L^p$ estimates.

\begin{prop} \label{broadtoreg} Suppose that for all $K, \eps$, the operator $\EP$ obeys the $k$-broad inequality:

\begin{equation} \label{btrbroad} \| \EP f \|_{\BLpkA(B_R)} \lesssim_{K, \eps} R^\eps \| f \|_{L^q}.  \end{equation}

(Here the quantities $k, A, p,q$, are fixed, and the inequality holds for all $R$.)

If $p \le q \le \infty$ and $p$ is in the range

\begin{equation} \label{btrp} 2 \cdot \frac{2n - k + 2}{2n - k} \le p \le 2 \cdot \frac{k-1}{k-2}, \end{equation}

then $ \EP$ obeys

\begin{equation} \label{btrreg} \| \EP f \|_{L^p(B_R)} \lesssim_\eps R^\eps \| f \|_{L^q}. \end{equation}

\end{prop}

Remarks.  The lower bound on $p$ is important.  The upper bound is less important, and it could probably be improved.  

Theorem \ref{broadEP} together with Proposition \ref{broadtoreg} implies Theorem \ref{L^pEP}.  If $n$ is even, we use $k = \frac{n}{2} + 1$.  By Theorem \ref{broadEP}, we have $\| \EP f \|_{\BLpkA(B_R)} \lesssim R^\eps \| f \|_{L^2} \lesssim R^\eps \| f \|_{L^p}$ for $ p$ slightly bigger than $\pkn = 2 \cdot \frac{n+k}{n+k-2}$.  With our choice of $k$, we also have $\pkn = 2 \cdot \frac{2n - k + 2}{2n - k}$.   Applying Proposition \ref{broadtoreg}, we get $\| \EP f \|_{L^p(B_R)} \lesssim R^\eps \| f \|_{L^p}$ for $p$ slightly bigger than $\pkn$.  Interpolating with the trivial $L^\infty$ bound gives this estimate for all $p > \pkn$.  Finally, applying $\eps$-removal (\cite{T1}) gives Theorem \ref{L^pEP}.  If $n$ is odd, we use $k = \frac{n+1}{2}$.  The argument is the same, (but in this case, $\pkn > 2 \cdot \frac{2n - k + 2}{2n - k}$).

\begin{proof} By hypothesis, we have an inequality of the form

\begin{equation} \label{broadhyp} \sum_{B_{K^2} \subset B_R} \min_{V_1, ..., V_A} \max_{\tau \notin V_a} \int_{B_{K^2}} | \EP f_\tau |^p \le C(K, \eps) R^{p \eps} \| f \|_{L^q}^p. \end{equation}

Here $V_1, ..., V_A$ are $(k-1)$-planes, and we write $\tau \notin V_a$ as an abbreviation for $\Angle( G(\tau), V_a ) > K^{-1}$.  

For each $B_{K^2}$, we fix a choice of $V_1, ..., V_A$ achieving the minimum above.  Then we can write

\begin{equation} \label{broadnarrow} \int_{B_{K^2}} | \EP f|^p \lesssim K^{O(1)} \max_{\tau \notin V_a} \int_{B_{K^2}} | \EP f_\tau|^p + \sum_{a=1}^A \int_{B_{K^2}} \left| \sum_{\tau \in V_a} \EP f_\tau \right|^p. \end{equation}

The first term is the ``broad'' part, and it can be controlled by the $k$-broad estimate.  We handle the second term, the ``narrow'' part, by a decoupling-type argument.  We work with $B_{K^2}$ so that we can cleanly apply the decoupling theorem from \cite{B4}.  (The paper \cite{BG} contains a different but closely related argument.) 

\begin{theorem} (\cite{B4}) \label{decthmbour} Suppose that $g: \RR^m \rightarrow \CC$ with $\hat g$ supported in the $K^{-2}$-neighborhood of the truncated paraboloid.  Divide this neighborhood into slabs $\tau$ with $m-1$ long directions of length $K^{-1}$ and one short direction of length $K^{-2}$.  Write $g = \sum_\tau g_\tau$, where $\hat g_\tau = \chi_\tau \hat g$.  Then on any ball of radius $K^2$, for $2 \le p \le 2 \cdot \frac{m}{m-1}$, 

\begin{equation} \label{decoupling} \| g \|_{L^p(B_{K^2})} \lesssim_\delta K^\delta \left( \sum_\tau \| g_\tau \|_{L^p(W_{B_{K^2}})}^2 \right)^{1/2}, \end{equation}

\noindent where $W_{B_{K^2}}$ is a weight measure, approximately the volume measure on $B_{K^2}$ and rapidly decaying.
\end{theorem}

Applying this decoupling estimate with $m= k-1$, we get the following lemma:

\begin{lemma} \label{decoupnarrow} 

$$ \left( \int_{B_{K^2}} \left| \sum_{\tau \in V_a} \EP f_\tau \right|^p \right)^{1/p} \lesssim_\delta K^\delta \left[ \sum_{\tau \in V_a} \left( \int W_{B_{K^2}} | \EP f_\tau |^p \right)^{2/p} \right]^{1/2}. $$
\end{lemma}

\begin{proof} On $B_{K^2}$, we use coordinates $(u, v)$ where $v$ is parallel to $V_a$ and $u$ is perpendicular to $V_a$.  We write $B_{u, K^2}$ for a ball of radius $K^2$ in the $u$-coordinates and $B_{v, K^2}$ for a ball of radius $K^2$ in the $v$ coordinates.
If we restrict $\EP f_\tau$ to the $k-1$-plane $\{ u \} \times \RR^{k-1}$ (parallel to $V_a$), then its Fourier transform is supported in the $K^{-2}$ neighborhood of a cap $\tau'$ in the $K^{-2}$-neighborhood of a paraboloid.  By Theorem \ref{decthmbour}, we get

\begin{equation*}
 \left\|  \sum_{\tau \in V_a} \EP f_\tau \right\|_{L^p(\{ u \} \times B_{v, K^2})} \lesssim_\delta K^\delta \left( \sum_{\tau \in V_a} \| \EP f_\tau \|_{L^p(\{ u \} \times W_{B_{v, K^2}})}^2 \right)^{1/2}. \end{equation*}

Using Fubini and Minkowski, we then get

$$ \left\|  \sum_{\tau \in V_a} \EP f_\tau \right\|_{L^p(B_{u, K^2} \times B_{v, K^2})} \lesssim_\delta K^\delta \left( \sum_{\tau \in V_a} \| \EP f_\tau \|_{L^p(W_{B_{K^2}})}^2 \right)^{1/2}. $$

\end{proof}

The number of $\tau \in V_a$ is $\lesssim K^{k-2}$.  Applying Holder's inequality, we see that

\begin{equation} \label{decnarrow2} \int_{B_{K^2}} \left| \sum_{\tau \in V_a} \EP f_\tau \right|^p
\lesssim_\delta K^\delta K^{(k-2)(\frac{p}{2} - 1)} \sum_{\tau \in V_a}  \int W_{B_{K^2}} | \EP f_\tau |^p. 
\end{equation}

At this point, we have gotten as much as we can from the knowledge that $\tau \in V_a$, and we relax this estimate to

\begin{equation} \label{decnarrow3} \int_{B_{K^2}} \left| \sum_{\tau \in V_a} \EP f_\tau \right|^p
\lesssim_\delta K^\delta K^{(k-2)(\frac{p}{2} - 1)} \sum_{\tau}  \int W_{B_{K^2}} | \EP f_\tau |^p. 
\end{equation}

Next we sum this inequality over all $a = 1, ..., A$ and over all $B_{K^2} \subset B_R$.  We let $W = \sum_{B_{K^2} \subset B_R} W_{B_{K^2}}$.  

\begin{equation} \label{decnarrowsum} 
\sum_{B_{K^2} \subset B_R} \sum_{a=1}^A \int_{B_{K^2}} \left| \sum_{\tau \in V_a} \EP f_\tau \right|^p \lesssim K^\delta K^{(k-2)(\frac{p}{2} - 1)} \sum_\tau \int W | \EP f_\tau|^p.  \end{equation}

We note that $W \lesssim 1$ on $B_{2R}$ and $W \le \RD(R)$ outside $B_{2R}$.  
Therefore we get

\begin{equation} \label{decnarrowsum2} 
\sum_{B_{K^2} \subset B_R} \sum_{a=1}^A \int_{B_{K^2}} \left| \sum_{\tau \in V_a} \EP f_\tau \right|^p \lesssim K^\delta K^{(k-2)(\frac{p}{2} - 1)} \sum_\tau \int_{B_{2R}} | \EP f_\tau|^p + \RD(R) \| f \|_{L^q}^p.  \end{equation}

Combining this estimate for the narrow part with our estimate for the broad part, we have

\begin{equation} \label{broad+dec}
\int_{B_R} | \EP f|^p \le C(K, \eps) R^{p \eps} \| f \|_{L^q}^p + C K^\delta K^{(k-2)(\frac{p}{2} - 1)} \sum_\tau \int_{B_{2R}} | \EP f_\tau|^p.
\end{equation}

With this inequality in hand, we will prove (\ref{btrreg}) by induction on the radius, where we use the induction assumption in order to handle the contribution of the $f_\tau$ terms.  The inequality we wish to prove is

\begin{equation} \label{btrgoal}
\int_{B_R} | \EP f|^p \le \bar C(\eps) R^{p \eps} \| f \|_{L^q}^p.
\end{equation}

By induction on the radius we can assume that (\ref{btrgoal}) holds for radii less than $R/2$.  We use this induction and parabolic rescaling to handle the contribution of each $f_\tau$.  

On the ball $\tau$ we introduce new coordinates.  Let $\omega_\tau$ be the center of $\tau$, and recall that the radius of $\tau$ is $K^{-1}$.  Then we introduce a new coordinate $\tilde \omega \in B^{n-1}$, by

\begin{equation} \label{defomegatilde} \tilde \omega = K ( \omega - \omega_\tau). \end{equation}

We rewrite the phase in these coordinates:

$$ x_1 \omega_1 + ... + x_{n-1} \omega_{n-1} + x_n | \omega_n |^2 = 
\Fcn(x) + \tilde x_1 \tilde \omega_1 + ... + \tilde x_{n-1} \tilde \omega_{n-1} + \tilde x_n | \tilde \omega|^2, $$

where $\Fcn(x)$ denotes a function of $x$ only and

$$ \tilde x_j = K^{-1} (x_j + 2 \omega_{\tau,j} x_n) \textrm{ for } 1 \le j \le n-1; \tilde x_n = K^{-2} x_n. $$

Here $\omega_{\tau,j}$ denotes the $j^{th}$ coordinate of $\omega_\tau$.  Note that the linear transformation $x \mapsto \tilde x$ sends $B_R$ into $B_{C R K^{-1}}$ and has Jacobian $\sim K^{-(n+1)}$.

We define 

$$\tilde f_\tau(\tilde \omega) = f_\tau (\omega) = f_\tau( K^{-1} \tilde \omega + \omega_\tau ), $$

so that

$$ |E f_\tau(x)| = K^{-(n-1)} | E \tilde f_\tau (\tilde x) |. $$

By induction on the radius, we can assume that (\ref{btrgoal}) holds for $\tilde f_\tau$ on a ball of radius $C R K^{-1}$:

$$ \int_{B_{C R K^{-1}}} | \EP \tilde f_\tau|^p \lesssim \bar C(\eps) R^{p \eps} K^{-p \eps} \| \tilde f_\tau \|_{L^q}^p. $$

By change of variables, we have

$$ \int_{B_{2R}} | \EP f_\tau|^p \lesssim K^{(n+1)} K^{-(n-1)p}  \int_{B_{C R K^{-1}}} | \EP \tilde f_\tau|^p \lesssim $$

$$ \lesssim \bar C(\eps) (R/K)^{p \eps} K^{(n+1) - (n-1)p } \| \tilde f_\tau \|_{L^q}^p =
\bar C(\eps) R^{p \eps} K^{(n+1) - (n-1)p - p \eps} K^{(n-1)\frac{p}{q}}  \| f_\tau \|_{L^q}^p. $$

Plugging this bound into (\ref{broad+dec}), we get

\begin{equation} \label{broad+dec+parab}
\int_{B_R} | \EP f|^p \le C(K, \eps) R^{p \eps} \| f \|_{L^q}^p + C \bar C(\eps) R^{p \eps} K^{(k-2)(\frac{p}{2} - 1) + (n+1) - (n-1)p - p \eps + \delta} K^{(n-1)\frac{p}{q}} \sum_\tau \| f_\tau \|_{L^q}^p. 
\end{equation}

There are $K^{n-1}$ different $\tau \subset B^{n-1}$. Since $p \le q$, we can apply Holder to see that

$$ \sum_\tau \| f_\tau \|_{L^q}^p \le \left( \sum_\tau \| f_\tau \|_{L^q}^q \right)^{p/q} K^{(n-1) ( 1 - \frac{p}{q})} = \| f \|_{L^q}^p K^{(n-1) ( 1 - \frac{p}{q})}. $$

\noindent Plugging this into the last inequality, we see that the dependence on $q$ drops out, and we are left with:

\begin{equation} \label{broad+dec+parab2}
\int_{B_R} | \EP f|^p \le C(K, \eps) R^{p \eps} \| f \|_{L^q}^p + C \bar C(\eps) R^{p \eps} K^{(k-2)(\frac{p}{2} - 1) + (n+1) - (n-1)p + (n-1) - p \eps + \delta} \| f \|_{L^q}^p. 
\end{equation}

We can close the induction as long as the exponent of $K$ is negative.  (First we choose $K$ large enough so that the second term is bounded by $(1/2) \bar C(\eps) R^{p \eps} \| f \|_{L^p}^p$.  Then we choose $\bar C(\eps)$ sufficiently large so that the first term is also bounded by $(1/2) \bar C(\eps) R^{p \eps} \| f \|_{L^p}^p$.)  Given $\eps > 0$, we can choose $\delta < \eps$.  So the induction closes as long as

$$ (k-2)\left(\frac{p}{2} - 1 \right) + (n+1) - (n-1)p + (n-1) \le 0. $$

This is equivalent with the lower bound for $p$ in our hypothesis (\ref{btrp}): $2 \cdot \frac{2n - k + 2}{2n - k} \le p$.
\end{proof}

\section{Appendix: Keeping track of parameters}

We have several small parameters.  In this appendix, we try to provide a reference to help the reader keep track of the parameters.  We list all the parameters, how they compare to each other, and where they appear in the argument.  

We begin with the small parameters.  For each $\eps > 0$, there is a sequence of small parameters 

$$ \dt \ll \delta_{n-1} \ll \delta_{n-2} \ll ... \ll \delta_1 \ll \delta_0 \ll \eps. $$

In this sequence, each parameter is far smaller than the next.  For instance, we will use that $\delta_{m} < \eps \delta_{m-1}$.  

The parameter $\dt$ appears in the wave packet decomposition.  The tubes $\Ttv$ in the wave packet decomposition have thickness $R^{1/2 + \dt}$.  The parameter $\dm$ appears in the $m$-dimensional case of the main inductive estimate, Proposition \ref{mainind}: in this estimate we suppose that $f = \sum_{\tv} \ftv$ is concentrated on wave packets that are $R^{-1/2 + \dm}$-tangent to $Z$ on the ball $B_R$.  

Another geometric parameter that appears is the radius $\rho$.   In the transverse algebraic case, we decompose $B_R$ into smaller balls $B_j$.  If we are working on tubes that cluster in the neighborhood of an $m$-dimensional variety $Z$, and are transverse to an $l$-dimensional variety $Y$, then the radius of each $B_j$ is $\rho$ given by

$$ \rho^{1/2 + \dl} = R^{1/2 + \dm}. $$

\noindent The quotient $R/\rho$ has size $R^{\theta(\dl)}$, which dominates $R^{O(\dm)}$. 

Then there are positive parameters $K, A$.  We have 

$$ 1 \ll A \ll K. $$

We need $A = A(\eps)$ sufficiently large to run the proof of Theorem \ref{broadEP}, and the broad inequality is most useful when $K$ is much larger than $A$.

Given $\eps$, we then fix the small parameters $\delta$ and the larger parameters $A, K$.  Then we consider $R \rightarrow \infty$.  In the statement of Theorem \ref{broadEP}, the constant depends on $\eps$ and on $K$.  By choosing this constant large enough, the theorem holds trivially unless $R$ is very large compared to all these fixed parameters.

\section{Further directions}

\subsection{Honest $k$-linear estimates} Our main result, Theorem \ref{broadEP} is a weak version of the $k$-linear restriction estimate from Conjecture \ref{klinear}.  For some purposes, we have seen that Theorem \ref{broadEP} is a good substitute for a $k$-linear restriction estimate, but there are surely other situations where an honest $k$-linear estimate is better.  When I tried to prove Conjecture \ref{klinear} using this method, I ran into the following problem.  There are $k$ different functions $\EP f_j$ to consider.  It may happen that for some of these $k$ functions, the wave packets of $\EP f_j$ are tangent to a variety $Z$, and for others of these $k$ functions, the wave packets of $\EP f_j$ are transverse to the variety $Z$.  I didn't find a good way to deal with this scenario.  The $k$-broad norm, $\BLpkA$, is designed to get around this situation.

\subsection{Kakeya-type estimates for low degree varieties}

We now return to the extension operator for the paraboloid.  Theorem \ref{broadEP} gives essentially sharp broad $L^p$ estimates of the form:

\begin{equation} \label{furbroadl2} \| \EP f \|_{\BLpkA(B_R)} \lesssim R^\eps \| f \|_{L^2}. \end{equation}

\noindent We have seen that this estimate holds if and only if $p \ge \pkn$.  What if we consider other norms on the right-hand side?  For some $q$ larger than 2, can we prove an estimate 

\begin{equation} \label{furbroadlq} \| \EP f \|_{\BLpkA(B_R)} \lesssim R^\eps \| f \|_{L^q}, \end{equation}

\noindent for some $p < \pkn$?

In the introduction, we mentioned some sharp examples for Theorem \ref{broadEP}.  The first question to ask ourselves is whether an inequalitiy of the form (\ref{furbroadlq}) may hold for these examples.  In the examples we considered, the wave packets $\EP \ftv$ concentrate in the $R^{1/2}$-neighborhood of a low degree variety $Z$.  Let us consider the set of caps $\theta$ that can appear in such an example.  Define $\Theta(Z)$ as

$$ \Theta(Z) := \{ \theta \textrm{ so that for some } v, \Ttv \subset N_{R^{1/2 + \dt}}(Z) \}, $$

\noindent  In such an example, the function $f$ must be supported in $ \cup_{\theta \in \Theta(Z)} \theta$. If the volume of this union is much less than 1, then for $q> 2$, $\| f \|_{L^q}$ will be much bigger than $ \| f \|_{L^2}$, and so our special class of examples will obey an inequality of the form (\ref{furbroadlq}). 

In fact, if we had good estimates for $| \Theta(Z) |$, then I believe we could input them into the proof of Theorem \ref{broadEP} to get some further estimates of the form (\ref{furbroadlq}), roughly following the argument in \cite{Gu4}.  

If $Z$ is an $m$-dimensional plane, then it is easy to check that $| \Theta(Z) | \sim (R^{1/2})^{m-1}$ and so $|\Omega(Z) | \sim (R^{1/2})^{m-1} \cdot (R^{1/2})^{-(n-1)}$.  It seems reasonable to conjecture that a similar bound holds for any $m$-dimensional variety $Z$ of small degree:

\begin{conjecture} \label{kakalg} If $Z$ is an $m$-dimensional variety in $\RR^n$ of degree at most $D$, then

\begin{equation} \label{kakalgeq} | \Theta(Z) | \le C(n, D, \eps) (R^{1/2})^{m-1 + \eps}. \end{equation}

\end{conjecture}

Conjecture \ref{kakalg} is a very special case of the Kakeya conjecture.  One variant of the Kakeya conjecture goes as follows:

\begin{conjecture} \label{kakgen} (Kakeya conjecture) Suppose that $X \subset B^n(1)$.  Suppose that $T_i \subset X$ are tubes of length 1 and radius $\delta$, pointing in $\delta$-separated directions.  Then

$$ \textrm{Number of tubes} \le C(\eps) \delta^{-\eps} \frac{ \Vol(X)}{\Vol( \textrm{tube} )} . $$

\end{conjecture}

(I haven't seen this exact version of the Kakeya conjecture in print before, but it's straightforward to check that the maximal function version of the Kakeya conjecture implies Conjecture \ref{kakgen}, which implies the Minkowski dimension version of the Kakeya conjecture.)  
Now Conjecture \ref{kakalg} is just the special case of Conjecture \ref{kakgen} where the set $X$ is the $\delta$-neighborhood of a low degree algebraic variety.

Conjecture \ref{kakalg} also came up in ongoing joint work with Josh Zahl on the Kakeya problem in $\RR^4$.  
I think it is a basic issue that comes up in trying to apply polynomial methods to the restriction problem or the Kakeya problem.


\begin{thebibliography}{5}

\vskip.125in

\bibitem[Be]{Be} I. Bejenaru, The optimal trilinear restriction estimate for a class of hypersurfaces with curvature, arXiv:1603.02965

\bibitem[Be2]{Be2} I. Bejenaru, Optimal multilinear restriction estimates for a class of surfaces with curvature,  arXiv:1606.02634

\bibitem[B1]{B1} J. Bourgain, Besicovitch type maximal operators and applications to Fourier analysis. Geom. Funct. Anal. 1 (1991), no. 2, 147-187. 

\bibitem[B2]{B2} J. Bourgain,  $L^p$-estimates for oscillatory integrals in several variables. Geom. Funct. Anal. 1 (1991), no. 4, 321-374.

\bibitem[B3]{B3} J. Bourgain, Some new estimates on oscillatory integrals, Annals Math. St. 42, Princeton UP (1995), 83-112.

\bibitem[B4]{B4} J. Bourgain, Moment inequalities for trigonometric polynomials with spectrum in curved hypersurfaces. Israel J. Math. 193 (2013), no. 1, 441-458. 

\bibitem[BD]{BD} J. Bourgain, C. Demeter, The proof of the $l^2$ decoupling conjecture. Ann. of Math. (2) 182 (2015), no. 1, 351-389.

\bibitem[BG]{BG}  J. Bourgain, L. Guth, Bounds on oscillatory integral operators based on multilinear estimates. Geom. Funct. Anal. 21 (2011), no. 6, 1239-1295. 

\bibitem[BCT]{BCT} J. Bennett, A. Carbery, and T. Tao, On the multilinear restriction and Kakeya conjectures. Acta Math. 196 (2006), no. 2, 261-302.

\bibitem[CKW]{CKW}  X. Chen, N. Kayal, and A. Wigderson, Partial derivatives in arithmetic complexity and beyond. Found. Trends Theor. Comput. Sci. 6 (2010), no. 1-2, pages 1-138 (2011).

\bibitem[CEGSW]{CEGSW} K.L. Clarkson, H. Edelsbrunner, L. Guibas, M Sharir, and E. Welzl,
{\it Combinatorial Complexity bounds for arrangements of curves and spheres}, Discrete Comput. Geom.
(1990)  5,  99-160.


\bibitem[D]{D} Z. Dvir, On the size of Kakeya sets in finite fields. J. Amer. Math. Soc. 22 (2009), no. 4, 1093-1097.

\bibitem[GP]{GP} V. Guillemin and A. Pollack, { \it Differential Topology}, AMS Chelsea Publishing, 1974, reprinted 2010.



\bibitem[G]{Gu4} L. Guth, A restriction estimate using polynomial partitioning. J. Amer. Math. Soc. 29 (2016), no. 2, 371-413. 

\bibitem[GHI]{GHI}, L. Guth, J. Hickman, M. Iliopoulou, Oscillatory integral operators and polynomial partitioning, preprint 

\bibitem[GK]{GK} L. Guth and N. Katz, On the Erd{\H o}s distinct distances problem in the plane,  Ann. of Math. (2) 181 (2015), no. 1, 155-190. 

\bibitem[H]{H} L. Hormander, Oscillatory integrals and multipliers on $FL^p$, Arkiv Math. II (1973), 1-11.


\bibitem[L]{L} S. Lee, Linear and bilinear estimates for oscillatory integral operators related to restriction to hypersurfaces, J. Funct. Anal. 241:1 (2006), 56-98.





\bibitem[OW]{OW} Y. Ou and H. Wang, A cone restriction estimate using polynomial partitioning,  arXiv:1704.05485

\bibitem[S]{Ste} E. Stein, Some problems in harmonic analysis. Harmonic analysis in Euclidean spaces (Proc. Sympos. Pure Math., Williams Coll., Williamstown, Mass., 1978), Part 1, pp. 3-20, Proc. Sympos. Pure Math., XXXV, Part, Amer. Math. Soc., Providence, R.I., 1979.

\bibitem[S2]{St2} E. Stein, Oscillatory integrals in Fourier analysis, Beijing lectures in Harmonic Analysis, Annals Math. St. 112, Princeton UP (1986).

\bibitem[StTu]{ST} A. Stone, and J. Tukey, Generalized "sandwich'' theorems.
Duke Math. J. 9, (1942) 356-359. 


\bibitem[TVV]{TVV}  T. Tao, A. Vargas, L. Vega, A bilinear approach to the restriction and Kakeya conjectures. J. Amer. Math. Soc. 11 (1998), no. 4, 967-1000.

\bibitem[T1]{T1} T. Tao, The Bochner-Riesz conjecture implies the restriction conjecture.
Duke Math. J. 96 (1999), no. 2, 363-375. 

\bibitem[T2]{T2} T. Tao,  A sharp bilinear restrictions estimate for paraboloids. Geom. Funct. Anal. 13 (2003), no. 6, 1359-1384. 



\bibitem[Wi]{Wi} L. Wisewell, Kakeya sets of curves, Geom. Funct. Anal. 15 (2005), no. 6, 1319-1362.

\bibitem[W1]{W1} T. Wolff, A sharp bilinear cone restriction estimate. Ann. of Math. (2) 153 (2001), no. 3, 661-698.



\bibitem[W2]{W4} T. Wolff, Local smoothing type estimates on $L^p$ for large p. Geom. Funct. Anal. 10 (2000), no. 5, 1237-1288. 

\bibitem[W3]{W5} T. Wolff, A Kakeya-type problem for circles. Amer. J. Math. 119 (1997), no. 5, 985-1026. 

\bibitem[Wo]{W}  R. Wongkew, Volumes of tubular neighbourhoods of real algebraic varieties. Pacific J. Math. 159 (1993), no. 1, 177-184.


\end{thebibliography}
\end{document}